\documentclass[12pt]{amsart}
\usepackage{epstopdf}
\usepackage[numbers,sort&compress]{natbib}
\bibpunct[, ]{[}{]}{,}{n}{,}{,}
\makeatletter
\def\NAT@def@citea{\def\@citea{\NAT@separator}}
\makeatother
\usepackage{xcolor}

\usepackage{amsmath}
\usepackage{amssymb}
\usepackage{amsthm}
\usepackage{amsfonts}
\usepackage{arydshln}
\usepackage[thinlines]{easybmat}
\numberwithin{equation}{section}
\DeclareMathOperator{\tr}{\text{\rm tr}}
\DeclareMathOperator{\sees}{\text{\rm ess sup}}
\newcommand{\h}{\mathcal{H}}

\newcommand{\Kt}{K_\Theta}

\def\f{{\mathbf{f}}}

\def\g{{\bf g}}

\newtheorem{claim}{claim}[section]
\newtheorem{theorem}[claim]{Theorem}

\newtheorem{lemma}[claim]{Lemma}
\newtheorem{proposition}[claim]{Proposition}

\theoremstyle{definition}

\newtheorem{remark}[claim]{Remark}
\newtheorem{example}[claim]{Example}

\title{Conjugations in $L^2(\h)$}
\author[M. C. C\^amara]{M. Cristina C\^amara}
\address{M. Cristina C\^amara, Center for Mathematical Analysis, Geometry and Dynamical Systems, Mathematics Department, Instituto Superior T\'ecnico, Universidade de Lisboa, Av. Rovisco Pais, 1049--001 Lisboa, Portugal}
\email{cristina.camara@tecnico.ulisboa.pt}

\author[K. Kli\'s--Garlicka]{Kamila Kli\'s--Garlicka}
\address{Kamila Kli\'s-Garlicka, Department of Applied Mathematics, University of Agriculture, ul. Balicka 253c, 30-198 Krak\'ow, Poland}
\email{rmklis@cyfronet.pl}
\author[B. \L anucha]{Bartosz \L anucha}
\address{Bartosz \L anucha, Department of Mathematics,  Maria Curie-Sk\l odowska University, Maria Curie-Sk\l o\-dow\-ska Square 1, 20-031 Lublin, Poland}
\email{bartosz.lanucha@poczta.umcs.lublin.pl}
\author[M. Ptak]{Marek Ptak}
\address{Marek Ptak, Department of Applied Mathematics, University of Agriculture, ul. Balicka 253c, 30-198 Krak\'ow, Poland}
\email{rmptak@cyf-kr.edu.pl}

\thanks{The work of the first author was partially supported by FCT/Portugal through
UID/MAT/04459/2013. The research of the second and the fourth authors was financed by the Ministry of Science and Higher Education of the Republic of Poland}

\begin{document}
\begin{abstract}{Conjugations commuting with $\mathbf{M}_z$ and intertwining $\mathbf{M}_z$ and $\mathbf{M}_{\bar z}$ in $L^2(\h)$, where  $\h$ is a Hilbert space, are characterized. We also investigate which of them leave invariant the whole Hardy space $H^2(\h)$ or a model space $\Kt=H^2(\h)\ominus\Theta H^2(\h)$, where $\Theta$ is a pure operator valued inner function.}
\end{abstract}
\maketitle

\section{Introduction}
The motivation to study conjugations (i.e., antilinear isometric involutions) has its roots in physics (\cite{GPP}), in particular in non-hermitian quantum mechanics and spectral analysis of complex symmetric operators.
There are many  important examples of complex symmetric operators, that is $C$--symmetric operators with respect to some conjugation $C$, namely normal operators, Hankel operators, truncated Toeplitz operators (see for example \cite{GP, GPP, CGW, GMR, GP2, CKP, CFT, KoLee18, Sarason}).

In \cite{CKP, CKLP} all conjugations in the classical $L^2$ space on the unit circle commuting with $M_z$ or intertwining the operators $M_z$ and $M_{\bar z}$ (in other words, all conjugations $C$ according to which the operator $M_z$ is $C$--symmetric, see the definition below) were fully characterized. The behaviour of such conjugations was also studied in connection with an analytic part of the space $L^2$  and model spaces, in particular there were characterized all conjugations leaving the whole Hardy space and  model spaces invariant. In what follows we study similar questions concerning conjugations in  $L^2$ spaces with values in a certain Hilbert space $\h$. The investigation in this direction is important for its relation with  B. Sz.-Nagy--C. Foia\c{s}  theory \cite[Chap.6]{NF} saying that $C_0$ contractions with finite defect indexes are unitarily equivalent to multiplication by the independent variable in a certain model space given by an operator valued inner function. In other words the results from the paper can be moved by unitary equivalence to contractions on Hilbert spaces, keeping suitable assumptions.

Denote by $\h$ a complex Hilbert space, by $L(\h)$ the algebra of all linear bounded operators on $\h$ and by $LA(\h)$ the space of all bounded antilinear operators on $\h$.
A {\it conjugation} $C$ in $\h$  is an antilinear isometric involution, i.e., $C^2=I_{\h}$ and
\begin{equation}
\langle Cf,Cg\rangle = \langle g,f \rangle \quad \text{ for all } f,g\in\h.
\end{equation}
An operator $A\in L(\h)$ is called $C$--{\emph symmetric} if $CAC=A^*$.
Recall that for $A\in LA(\mathcal{H})$ there exists a unique antilinear operator $A^\sharp$, called the {\it antilinear adjoint} of $A$, defined by the equality
\begin{equation}\label{as}\langle Af,g\rangle=\overline{\langle f,A^\sharp g\rangle},\end{equation}
for all $f, g\in \mathcal{H}$. It is clear, see \cite{CKP}, that $C^\sharp=C$ for any conjugation $C$, $(AB)^\sharp=B^*A^\sharp$ and similarly $(BA)^\sharp=A^\sharp B^*$ for $A\in LA(\h)$, $B\in L(\h)$.

 Let $L^2=L^2(\mathbb{T}, m)$ and $L^\infty=L^\infty(\mathbb{T},m)$ where $\mathbb{T}$ is the unit circle and $m$ the normalized Lebesgue measure and let $H^2$ denote the classical Hardy space on the unit disc $\mathbb{D}$. For an inner function $\theta$  (i.e., $\theta\in L^\infty\cap H^2$ and $|\theta|=1$ a.e. on $\mathbb{T}$) one can define the {\it model space} $K_\theta=H^2\ominus \theta H^2$.

The most natural conjugation $\tilde J$ in $L^2$ is defined as $\tilde Jf=\bar f$, for $f\in L^2$. This conjugation has two natural properties:  the operator $M_z$ is $\tilde J$--symmetric, i.e., $M_z \tilde J=\tilde J M_{\bar z}$, and
$\tilde J$ maps an analytic function into a co-analytic one, i.e., $\tilde J H^2=\overline{H^2}$.
Another natural conjugation in $L^2$ is $J^{\star}f=f^{\#}, \ f^{\#}(z)=\overline{ f(\bar z)}.$ The conjugation $J^{\star}$ has a completely different behaviour: it commutes with multiplication  by $z$ ($M_z {J}^\star={J}^\star M_{ z}$) and leaves analytic functions invariant, {${J}^\star H^2={H^2}$.}

In Section 3 we recall some basic properties of vector and operator valued functions. In Section 4, for any {separable} Hilbert space $\h$, we naturally extend the definitions of the conjugations $\tilde{J}$ and $J^\star$ on $L^2$ to the conjugations $\widetilde{\mathbf{J}}$ and $\mathbf{J}^\star$ on a vector valued space $L^2(\h)$ keeping the same properties with respect to the multiplication by the independent variable, i.e.,
$\widetilde{\mathbf{J}}\mathbf{M}_z=\mathbf{M}_{\bar z}\widetilde{\mathbf{J}}, \quad \mathbf{J}^\star \mathbf{M}_z=\mathbf{M}_z \mathbf{J}^\star.$  However, one needs to fix some conjugation $J$ on $\h$ (in case of $L^2$ the natural conjugation in $\mathbb{C}$, $z\mapsto \bar z$, plays this role).
 Theorems \ref{8.3} and \ref{8.1}  characterize all {\it $\mathbf{M}_z$--conjugations} $\mathbf{C}$ in $L^2(\h)$, i.e.,
\begin{equation}\label{mz}
{\mathbf{C}}\mathbf{M}_z=\mathbf{M}_{\bar z}{\mathbf{C}}
\end{equation}
and all {\it $\mathbf{M}_z$--commuting conjugations} in $L^2(\h)$, i.e.,
\begin{equation}\label{mc}
{\mathbf{C}}\mathbf{M}_z=\mathbf{M}_{z}{\mathbf{C}}.
\end{equation}
In Section 5, Theorem \ref{8.2} and Proposition \ref{5.5}, we describe all conjugations satisfying \eqref{mz} or \eqref{mc} and leaving the Hardy space $H^2(\h)$ invariant.
In Section 6 we study these conjugations for which vector valued model spaces are invariant. We concentrate on the case when the dimension of the underlying Hilbert space $\h$ is finite and the vector valued inner function is pure. The last section is devoted to those conjugations which leave shift invariant subspaces invariant. Generally, the vector valued case is much more complicated than the scalar one, see for example Theorems \ref{66}, \ref{th66}, \ref{75} and Example \ref{ex88}.
 Section 2 is devoted to the special case $\h=\mathbb{C}^2$. This  illustrates the general theorems and gives stronger result than in \cite{KoLee18}. On the other hand, it gives more precise characterizations (Theorem \ref{t3}), which are of independent interest.

\section{Conjugations in $L^2\oplus L^2$}
In Theorem 2.4 and Proposition 2.6 \cite{KoLee18} there are given the conditions for a $2\times 2$ operator matrix to be a conjugation. We will, however, use
equivalent conditions obtained by checking the antilinear selfadjointness and involutive property of a conjugation. Namely,
let $
\widetilde C=\left[   \begin{BMAT}{cc}{cc}
     D_1 & D_2 \\
     D_3 & D_4\\
   \end{BMAT} \right]$,
 where $D_j$ are nonzero antilinear operators on $\mathcal H$ for $j=1,2,3,4$. Then $\widetilde C$
is a conjugation on $\mathcal H \oplus \mathcal H$ if and only if the following conditions hold:
\begin{equation}\label{Ka}
D_1=D_1^\sharp,\,D_4=D_4^\sharp,\,D_3=D_2^\sharp,
\end{equation}
\begin{equation}\label{Kb}
D_1D_1^\sharp+D_2D_2^\sharp=I,\, D_2^\sharp D_2+D_4D_4^\sharp=I,
\end{equation}
\begin{equation}\label{Kc}
D_2^\sharp D_1+D_4D_2^\sharp=0. 
\end{equation}

Denote $\mathbf{M}_z=\left[   \begin{BMAT}{cc}{cc}
     M_z & 0 \\
     0 & M_z\\
   \end{BMAT} \right]$.
We will investigate the conditions for $\mathbf{C}$ to be an $\mathbf{M}_z$--conjugation and for it to commute  with $\mathbf{M}_z$.

In \cite[Theorem 2.4]{CKLP} all conjugations in $L^2$ commuting with $M_z$ were characterized. In particular it was shown that such a conjugation has to be of the form $M_\psi J^\star$ for some unimodular function $\psi\in L^\infty$ which is {\it symmetric}, i.e., $\psi(z)=\psi(\bar z)$ {a.e. on $\mathbb{T}$}. The following theorem gives a characterization of $\mathbf{M}_z$--commuting conjugations in $L^2\oplus L^2$.

\begin{theorem}\label{t2}
Let $\mathbf{C}$ be an antilinear operator on $L^2\oplus L^2$. Then
$
\mathbf{C}=\left[   \begin{BMAT}{cc}{cc}
     D_1 & D_2 \\
     D_2^\sharp & D_4\\
   \end{BMAT} \right]$ is a conjugation such that $\mathbf{M}_z \mathbf{C}=\mathbf{C}\mathbf{M}_z$ if and only if
   there are functions $\psi_i\in L^\infty$, $i=1,2,4$, such that $D_i=M_{\psi_i}J^\star$ and
   \begin{align}
   &\psi_1^{\#}=\overline\psi_1,\quad \psi_4^{\#}=\overline\psi_4,\label{14}\\
  & |\psi_1|^2=|\psi_4|^2=1-|\psi_2|^2, \label{15}
  \\ &\psi_1^{\#}\psi_2+\psi_2^{\#}\psi_4=0.\label{16}
   \end{align}
\end{theorem}
\begin{proof}
Easy calculations show that $\mathbf{M}_z \mathbf{C}=\mathbf{C}\mathbf{M}_z$ if and only if $M_z D_i =D_i M_z$ for $i=1,2,4$ and $M_z D^\sharp_2=D_2^\sharp M_z$. Hence for $i=1,2,4$ we have
 $$M_z D_i J^\star= D_i M_z J^\star=D_i J^\star M_z .$$
Thus the linear operators $ D_i J^\star$ commute with $M_z$, so they have to be of the form $D_i J^\star =M_{\psi_i}$ for $\psi_i\in L^\infty$. 

%
%
The condition \eqref{Ka} implies that for $i=1,4$,
$ M_{\psi_i}J^\star=J^\star M_{\overline\psi_i}$, i.e., $\psi_i^{\#}=\overline\psi_i$. This means that $\psi_1$ and $\psi_4$ are symmetric, i.e., $\psi_i(z)=\psi_i(\bar z)$ for $i=1,4$ {(a.e. on $\mathbb{T}$)}.
By \eqref{Kb}, for $i=1,4$, we get
$$(M_{\psi_i}J^\star)^2+M_{\psi_2}J^\star J^\star M_{\overline\psi_2}=I$$
which is equivalent to
$$\psi_i^{\#} \psi_i+|\psi_2|^2=1.$$
Finally, by \eqref{Kc} we get
\begin{align}
J^\star M_{\overline\psi_2} M_{\psi_1}J^\star+J^\star M_{\psi_4^{\#}} M_{\overline\psi_2^{\#}}J^\star&=0
\end{align}
which is equivalent to
\begin{align}
  \overline\psi_2\psi_1+\psi_4^{\#}\overline\psi_2^{\#} &=0.\label{aa2}
\end{align}
Taking into consideration the fact that $\psi_1$ and $\psi_4$ are symmetric we obtain \eqref{16}. 
\end{proof}
\begin{remark}
Note that if $\psi_2=0$, then conditions \eqref{14}--\eqref{16} imply that
 $D_1$ and $D_4$ are conjugations which commute with $M_z$. Hence by \cite[Theorem 2.4]{CKLP} we get $D_1=M_{\psi_1}J^\star$, $D_4=M_{\psi_4}J^\star$. On the other hand, this is also a consequence of Theorem \ref{t2}.

If now $\psi_1=0$ (which is equivalent to $\psi_4=0$), then \eqref{15} implies that $|\psi_2|=1$.
\end{remark}

\begin{example}
The following conjugations satisfy Theorem \ref{t2}:
 $$
\mathbf{J}_1^\star=\left[   \begin{BMAT}{cc}{cc}
     J^\star & 0 \\
     0 & J^\star\\
   \end{BMAT} \right], \quad\text{or}\quad
   \mathbf{J}_2^\star=\left[   \begin{BMAT}{cc}{cc}
     0 & J^\star \\
     J^\star & 0\\
   \end{BMAT} \right]\quad\text{or}\quad
    \mathbf{C}=\frac{1}{\sqrt{2}}\left[   \begin{BMAT}{cc}{cc}
     J^\star & \phantom{-}J^\star \\
     J^\star & -J^\star\\
   \end{BMAT} \right]. $$
\end{example}
\begin{example}
   Let $\psi_1(z)=\psi_4(z)=\frac{1}{2}(\bar z+z)$ and $\psi_2(z)=\frac{1}{2i}(-\bar z+z)$. In other words, $\psi_1(e^{it})=\psi_4(e^{it})=\cos t$ and $\psi_2(e^{it})=\sin t$. Then $\psi_2^{\#}(z)=\frac{1}{2i}(\bar z-z)$, i.e., $\psi_2^{\#}(e^{it})=-\sin t$. Observe that such functions satisfy conditions \eqref{14}--\eqref{16}, so $\mathbf{C}=\left[   \begin{BMAT}{cc}{cc}
    M_{\psi_1} J^\star & M_{\psi_2} J^\star \\
    M_{\overline\psi_2^{\#}} J^\star & M_{\psi_4}J^\star\\
   \end{BMAT} \right]$ is a conjugation satisfying Theorem \ref{t2}. \end{example}

The characterization of all $M_z$--conjugations in $L^2$ was given in \cite{CKP}. It was proved that such conjugations are of the form $M_\psi \tilde J$, where $\psi\in L^\infty$, $|\psi|=1$. In the space $L^2(\mathbb{C}^2)$ the characterization is more complex.

\begin{theorem}\label{t3}
Let $\mathbf{C}$ be an antilinear operator on $L^2(\mathbb{C}^2)$. Then
$
\mathbf{C}=\left[   \begin{BMAT}{cc}{cc}
     D_1 & D_2 \\
     D_2^\sharp & D_4\\
   \end{BMAT} \right]$ is a conjugation such that $\mathbf{M}_z\mathbf{C}=\mathbf{C}\mathbf{M}_{\bar z}$ if and only if
   there are  functions $\psi_i\in L^\infty$, $i=1,2,4$, such that $D_i=M_{\psi_i}\tilde J$ and
   \begin{align}
  & |\psi_1|^2=|\psi_4|^2=1-|\psi_2|^2,\label{18}
  \\ &\overline\psi_1\psi_2+\overline\psi_2\psi_4=0.\label{19}
   \end{align}
\end{theorem}
\begin{proof}
Note that $\mathbf{M}_z\mathbf{C}=\mathbf{C}\mathbf{M}_{\bar z}$ if and only if
\begin{equation}\label{mz1}
M_z D_i =D_i M_{\bar z} \text{  for } i=1,2,4.\end{equation}
Hence for $i=1,2,4$ we have
\begin{equation*} M_z D_i\tilde J=D_i M_{\bar z}\tilde J=D_i \tilde J M_z .\end{equation*}
Thus the linear operators $D_i\tilde J$ commute with $M_z$, so they have to be of the form $ D_i\tilde J=M_{\psi_i}$ for $\psi_i\in L^\infty$. Hence $D_i=M_{\psi_i}\tilde J$ for $\psi_i\in L^\infty$.

%

Note that $\mathbf{C}$ has to satisfy the conditions \eqref{Ka}--\eqref{Kc}. The condition \eqref{Ka} is satisfied automatically since  $D_i=M_{\psi_i}\tilde J$ are antilinearly self-adjoint for $\psi_i\in L^\infty$.
By checking \eqref{Kb} we get
$$( M_{\psi_1}\tilde J)^2+ M_{\psi_2}\tilde J \tilde J M_{\overline\psi_2}=I \text{ and } ( M_{\psi_4}\tilde J)^2+ M_{\psi_2}\tilde J \tilde J M_{\overline\psi_2}=I,$$
which is equivalent to \eqref{18}.
Finally by \eqref{Kc} we get
\begin{align*}
\tilde J M_{\overline\psi_2} M_{\psi_1} \tilde J+ M_{\psi_4}\tilde J\tilde J M_{\overline\psi_2}&=0
\end{align*}
which is equivalent to \eqref{19}.
\end{proof}
\begin{remark}
Note that if \eqref{18}--\eqref{19} are satisfied and $\psi_2=0$, then $|\psi_1|=|\psi_4|=1$. Hence $D_1$ and $D_4$ are $M_z$--conjugations in $L^2$.
On the other hand, if $\psi_1=0$ (which is equivalent to $\psi_4=0$), then $|\psi_2|=1$, which implies that $D_2$ is an $M_z$--conjugation.
\end{remark}
\begin{example}
The following conjugations satisfy Theorem \ref{t3}:
 $$
\widetilde{\mathbf{J}}_1=\left[   \begin{BMAT}{cc}{cc}
     \tilde J & 0 \\
     0 & \tilde J\\
   \end{BMAT} \right], \quad\text{or}\quad
\widetilde{\mathbf{J}}_2=\left[   \begin{BMAT}{cc}{cc}
     0 & \tilde J \\
     \tilde J & 0\\
   \end{BMAT} \right],\quad\text{or}\quad
\mathbf{C}=\frac{1}{\sqrt{2}}\left[   \begin{BMAT}{cc}{cc}
     \tilde J & \phantom{-}\tilde J \\
     \tilde J & -\tilde J\\
   \end{BMAT} \right]. $$
\end{example}
\begin{example}
   Let $\psi_1(z)=-\psi_4(z)=\frac{1}{2i}(-\bar z+z)$ and $\psi_2(z)=\frac{1}{2}(\bar z+z)$. In other words, $\psi_1(e^{it})=-\psi_4(e^{it})=\sin t$ and $\psi_2(e^{it})=\cos t$. Observe that such functions satisfy conditions \eqref{18}--\eqref{19}, so $\mathbf{C}=\left[   \begin{BMAT}{cc}{cc}
    M_{\psi_1} \tilde J & M_{\psi_2} \tilde J \\
     M_{\psi_2} \tilde J & M_{\psi_4}\tilde J\\
   \end{BMAT} \right]$ is a conjugation satisfying Theorem \ref{t3}.
\end{example}

\section{Operator valued functions}

{Let $\h$ be a complex separable Hilbert space. Denote by $L^2(\h)$ the space of all (classes of) functions $\f\colon\mathbb{T}\to \h$ which are measurable and satisfy
\begin{equation}\label{norma}
 \int_{\mathbb{T}}\| \f(z)\|^2dm(z)<\infty
\end{equation}
(the norm $\|\cdot\|$ under the integral is the norm in $\h$). So $\f\in L^2(\h)$ is understood as a class (represented by $\f$) of all measurable functions satisfying \eqref{norma} and equal to $f$ on $\mathbb{T}$ a.e. with respect to $m$. Recall that the measurability of $\f$ means that $z\mapsto \|\f(z)\|$ is a measurable function (or that $z\mapsto \langle \f(z),x\rangle$ is measurable for every $x\in \h$, which due to separability of $\h$ is an equivalent definition).}

 {The space $L^2(\h)$ is a Hilbert space with the inner product given by
$$\langle \f,\g \rangle=\int \langle \f(z), \g(z)\rangle dm,\quad \f,\g\in L^2(\h),$$
where the inner product under the integral is the inner product in $\h$ (note that $L^2(\h)$ is also separable).}

 {Functions in $L^2(\h)$ have an expansion analogous to the Fourier expansion in $L^2$. Observe that if $\f \in L^2(\h)$, then for each $n\in\mathbb{Z}$ the linear functional $y\mapsto \overline{\int_{\mathbb{T}}\langle \f(z),y\rangle \overline{z}^n dm(z)}$ is bounded and so there exists $x_n\in\h$ such that
\begin{equation}\label{fourier}
\langle x_n,y\rangle=\int_{\mathbb{T}}\langle \f(z),y\rangle \overline{z}^n dm(z)\quad \text{for all }y\in\h.
\end{equation}
The element $x_n$ is called the {\it n-th Fourier coefficient of}\ \  $\f$. It turns out that $\f\in L^2(\h)$ can be expressed as $\f=\sum\limits_{n=-\infty}^\infty x_n e_n $, where $x_n$ is given by \eqref{fourier}, $e_n(z)=z^n$ and the series converges in the norm of $L^2(\h)$. Moreover, for $\f=\sum\limits_{n=-\infty}^\infty x_n e_n\in L^2(\h)$ and $\g=\sum\limits_{n=-\infty}^\infty y_n e_n\in L^2(\h)$ we have
$$\|\f\|^2=\sum\limits_{n=-\infty}^\infty \|x_n\|^2$$
and
$$\langle\f,\g\rangle=\sum\limits_{n=-\infty}^\infty \langle x_n,y_n\rangle,$$
and this correspondence between elements of $L^2(\h)$ and sequences $\{x_n \}_{n=-\infty}^{\infty}\subset\h$ such that $\sum\limits_{n=-\infty}^\infty \|x_n\|^2<\infty$ is one--to--one (see, e.g., \cite[pp. 46--48]{RR}).}

 {Denote by $H^2(\h)$  the subspace of $L^2(\h)$ consisting of those functions from $L^2(\h)$ whose Fourier coefficients with negative indices are 0, i.e.,
$$H^2(\h)=\left\{\f\in L^2(\h): \f=\sum_{n=0}^{\infty} x_n e_n\right\}.$$
Each $\f
\in H^2(\h)$ can be also identified with a function
$$\f(\lambda)=\sum\limits_{n=0}^\infty x_n\lambda^n,\quad\lambda\in\mathbb{D},$$
which is analytic in the unit disk $\mathbb{D}$. Thus $H^2(\h)$ can be seen as a subspace of $L^2(\h)$ or as a space of functions analytic in $\mathbb{D}$ and with values in $\h$. The boundary values on $\mathbb{T}$ can be then obtained through radial limits (here the radial functions converge to the boundary function in the $L^2(\h)$ norm). For more details see also \cite{berc,NF}.}

 {A function $\mathbf{F}\colon\mathbb{T}\to L(\h)$ is said to be measurable, if for every $x\in \h$ the function $z\mapsto \mathbf{F}(z)x$ is measurable. Let us denote by $L^\infty(L(\h))$ the space of (again, classes of) all such measurable essentially bounded
 functions $\mathbf{F}$ which are essentially bounded, i.e,
 $$\|\mathbf{F}\|_{\infty}=\sees_{z\in\mathbb{T}}\|\mathbf{F}(z)\|<\infty$$
 (where $\|\mathbb{F}(z)\|$ denotes the operator norm of $\mathbb{F}(z)\in L(\h)$). For each $\mathbf{F}\in L^\infty(L(\h))$ we define a bounded linear operator $M_\mathbf{F}$ on $L^2(\h)$: for $\f\in L^2(\h)$,
  \[(M_\mathbf{F}\,\f)(z)=\mathbf{F}(z)\f(z)\quad \text{a.e. on }\mathbb{T} .\]
In particular, for $\f\in L^2(\h)$, we have, a.e. on $\mathbb{T}$,
$$ (\mathbf{M}_z \f)(z)=z\f(z)\quad \text{and}\quad (\mathbf{M}_{\bar z}\f)(z)=\bar z\f(z),$$
that is, $\mathbf{M}_z=M_{z\mathbf{I}_{\h}}$ and $\mathbf{M}_{\bar z}=M_{\bar z\mathbf{I}_{\h}}$. For $\mathbf{F}\in L^\infty(L(\h))$ the adjoint function $\mathbf{F}^*$ is naturally defined as $\mathbf{F}^*(z)=\mathbf{F}(z)^*$ (a.e. on $\mathbb{T}$) and we also define $\mathbf{F}^{\#}$ by $(\mathbf{F}^{\#} )(z)=\mathbf{F}(\bar z)^*$ (a.e. on $\mathbb{T}$). Clearly, $M^*_\mathbf{F}=M_{\mathbf{F}^*}$. If $\mathbf{F}\colon \mathbb{T}\to L(\h)$ is a constant function, $\mathbf{F}(z)=F$, we will denote by $F$ its action on $L^2(\h)$.}

   {Similarly, we define a measurable function $\mathbf{F}\colon\mathbb{T}\to LA(\h)$ and denote by $L^\infty(LA(\h))$ the space of all measurable essentially bounded
 functions valued in the space $LA(\h)$. For each $\mathbf{C}\in L^\infty(LA(\h))$ we define a bounded antilinear operator $A_\mathbf{C}$ on $L^2(\h)$:  for $\f\in L^2(\h)$,
  \[(A_\mathbf{C}\,\f)(z)=\mathbf{C}(z)\f(z)\quad \text{a.e. on }\mathbb{T}.\]}
 {Assume now that $\dim\h<\infty$. Then $L(\h)$ endowed with the Hilbert-Schmidt norm can be seen as a Hilbert space and we can consider the space $L^2(L(\h))$ defined as above (recall that the inner product in $L(\h)$ is then given by $\langle A,B\rangle=\tr(B^*A)$ for $A,B\in L(\h)$, see \cite[pp. 86--93]{conway2}). In that case $L^{\infty}(L(\h))\subset L^2(L(\h))$ (since here $\dim L(\h)<\infty$ and the Hilbert-Schmidt norm and the operator norm are equivalent), so every $\mathbf{F}\in L^{\infty}(L(\h))$ admits a Fourier expansion
\begin{equation}\label{jeden}
\mathbf{F}=\sum_{n=-\infty}^{\infty}{F}_n e_n,
\end{equation}
with ${F}_n\in L(\h)$ for $n\in\mathbb{Z}$ (the series converges in the $L^2(L(\h))$ norm). Let
\begin{displaymath}
\begin{split}
H^{\infty}(L(\h))&=L^{\infty}(L(\h))\cap H^2(L(\h))\\
&=\left\{\mathbf{F}\in L^{\infty}(L(\h))\ \colon\ \mathbf{F}=\sum_{n=0}^{\infty}{F}_n e_n\right\}.
\end{split}
\end{displaymath}
As mentioned above, every $\mathbf{F}=\sum_{n=0}^{\infty}{F}_n e_n\in H^{\infty}(L(\h))$ can be then extended to a function analytic in $\mathbb{D}$ by the formula
$$\mathbf{F}(\lambda)=\sum_{n=0}^{\infty}{F}_n\lambda^n,\quad\lambda\in\mathbb{D},$$
and this extension is bounded, $\displaystyle{\sup_{\lambda\in\mathbb{D}}\|\mathbf{F}(\lambda)\|<\infty}$. Moreover, every such bounded analytic function has boundary values a.e. on $\mathbb{T}$ (radial limits in the $L^2(L(\h))$ norm) and the boundary function belongs to $H^{\infty}(L(\h))$. Therefore a function from $H^{\infty}(L(\h))$ can be seen as an element of $L^{\infty}(L(\h))$ or as a bounded analytic operator valued function in $\mathbb{D}$ (see \cite[p. 232]{berc}).}

\section{$\mathbf{M}_z$--commuting and $\mathbf{M}_z$--conjugations in $L^2(\h)$}

For an arbitrary conjugation $J$ in $\h$ let us define two conjugations $\widetilde{\mathbf{J}}$ and $\mathbf{J}^\star$ on $L^2(\h)$ given by
\begin{equation}\label{jt}
(\widetilde{\mathbf{J}} \f)(z)=J(\f(z)) {\quad \text{a.e. on }\mathbb{T}}
\end{equation}
and
\begin{equation}\label{jg}
(\mathbf{J}^\star \f)(z)=\f^{\#}(z)=J(\f(\bar z)) {\quad \text{a.e. on }\mathbb{T}.}
\end{equation}
 {Note that for $\f=\sum\limits_{n=-\infty}^{\infty} x_n e_n\in L^2(\h)$,
\begin{equation}\label{jdef}
\widetilde{\mathbf{J}}\f=\sum_{n=-\infty}^{\infty} J(x_{-n})e_n
\quad\text{ and }\quad \mathbf{J}^\star \f=\sum_{n=-\infty}^{\infty}J(x_n)e_n.
\end{equation}
Indeed, if $\displaystyle{\widetilde{\mathbf{J}}\f=\sum_{n=-\infty}^{\infty} y_n e_n}$ is the Fourier expansion of $\widetilde{\mathbf{J}}\f$, then by \eqref{fourier}, for each $n\in\mathbb{Z}$ and $y\in \h$,
\begin{displaymath}
\begin{split}
\langle y_n,y\rangle&=\int_{\mathbb{T}}\langle (\widetilde{\mathbf{J}}\f)(z),y\rangle \overline{z}^n dm(z)=\int_{\mathbb{T}}\langle J(\f(z)),y\rangle \overline{z}^n dm(z)\\
&=\overline{\int_{\mathbb{T}}\langle \f(z),Jy\rangle {z}^n dm(z)}=\overline{\langle x_{-n},Jy\rangle}=\langle J(x_{-n}),y\rangle
\end{split}
\end{displaymath}
and so $y_n= J(x_{-n})$. 
Similarly, if $\displaystyle{\mathbf{J}^\star\f=\sum_{n=-\infty}^{\infty}y_n e_n}$, then for each $n\in\mathbb{Z}$ and $y\in \h$,
\begin{displaymath}
\begin{split}
\langle y_n,y\rangle&=\int_{\mathbb{T}}\langle (\mathbf{J}^\star\f)(z),y\rangle \overline{z}^n dm(z)=\int_{\mathbb{T}}\langle J(\f(\overline{z})),y\rangle \overline{z}^n dm(z)\\
&=\overline{\int_{\mathbb{T}}\langle \f(z),Jy\rangle \overline{z}^n dm(z)}=\overline{\langle x_{n},Jy\rangle}=\langle J(x_{n}),y\rangle
\end{split}
\end{displaymath}
and $y_n= J(x_{n})$. Observe that by \eqref{jdef},
$$\mathbf{J}^\star(H^2(\h))\subset H^2(\h) \quad\text{and}\quad\widetilde{\mathbf{J}}(H^2(\h))\subset H^2(\h)^{\perp}.$$
Moreover, conjugations $\widetilde{\mathbf{J}}$ and $\mathbf{J}^\star$ have the following properties:}
\begin{proposition}\label{js} For $\f\in L^2(\h)$ we have
\begin{enumerate}
\item $(\widetilde{\mathbf{J}}\,\mathbf{J}^\star \f)(z)=(\mathbf{J}^\star\, \widetilde{\mathbf{J}}\f)(z)=\f(\bar z)$  {for almost all $z\in\mathbb{T}$,}
\item $\widetilde{\mathbf{J}}\,\mathbf{M}_z=\mathbf{M}_{\bar z}\,\widetilde{\mathbf{J}}$,
\item $\mathbf{J}^\star\,\mathbf{M}_z=\mathbf{M}_z\,\mathbf{J}^\star$.
\end{enumerate}
\end{proposition}
\begin{proposition}\label{p2}
Let $J$ be a conjugation in $\h$, and let $\mathbf{F}\in L^\infty(L(\h))$. Then
\begin{enumerate}
\item $M_{\mathbf{F}}$ is $\widetilde{\mathbf{J}}$--symmetric if and only if $\mathbf{F}(z)$ is $J$--symmetric for almost all $z\in \mathbb{T}$.
\item  $M_{\mathbf{F}}$ is $\mathbf{J}^\star$--symmetric if and only if $J\mathbf{F}(z)J=\mathbf{F}^{\#}(z)$ for almost all $z\in \mathbb{T}$.
\item $\widetilde{\mathbf{J}}M_{\mathbf{F}}\widetilde{\mathbf{J}}=M_{\mathbf{F}^*}$ if and only if
 $\mathbf{J}^\star M_{\mathbf{F}}\mathbf{J}^\star=M_{\mathbf{F}^{\#}}$.
\item  $\widetilde{\mathbf{J}}M_{\mathbf{F}}\widetilde{\mathbf{J}}=M_{\mathbf{F}^{\#}}$ if and only if
 $\mathbf{J}^\star M_{\mathbf{F}}\mathbf{J}^\star=M_{\mathbf{F}^*}$.
\item if $\mathbf{F}(\bar z)=\mathbf{F}(z)$ for almost all $z\in \mathbb{T}$ , then  $M_{\mathbf{F}}$ is $\widetilde{\mathbf{J}}$--symmetric if and only if it is $\mathbf{J}^\star$--symmetric.
\end{enumerate}
\end{proposition}
\begin{proof}
Note that for $\f\in L^2(\h)$ we have
\begin{equation}\label{ej1}
\begin{split}
(\widetilde{\mathbf{J}}M_{\mathbf{F}}\widetilde{\mathbf{J}}\f)(z)&=J((M_{\mathbf{F}}\widetilde{\mathbf{J}}\f)(z))\\
&=J(\mathbf{F}(z)(\widetilde{\mathbf{J}}\f)(z))=J\mathbf{F}(z)J(\f(z)).
\end{split}
\end{equation}
Hence (1) is proved.
Similarly, (2) follows from the equality
\begin{equation}\label{ej2}
\begin{split}
({\mathbf{J}}^\star M_{\mathbf{F}}{\mathbf{J}}^\star \f)(z)&=J((M_{\mathbf{F}}{\mathbf{J}^\star}\f)(\bar z))
\\&=J(\mathbf{F}(\bar z)({\mathbf{J}}^\star \f)(\bar z))=J\mathbf{F}(\bar z)J(\f(z)).
\end{split}
\end{equation}
Comparing \eqref{ej1} with \eqref{ej2} we get (3) and (4). Condition (5) follows from (3), since $\mathbf{F}(\bar z)=\mathbf{F}(z)$  {(a.e. on $\mathbb{T}$)} if and only if $\mathbf{F}(z)^*=\mathbf{F}^{\#}(z)$  {(a.e. on $\mathbb{T}$)}.
\end{proof}

The following theorem gives a characterization of all $\mathbf{M}_z$--commuting conjugations in $L^2(\h)$.
\begin{theorem}\label{8.1}
Let $J$ be a conjugation  in $\h$. Then the following conditions are equivalent:
\begin{enumerate}
\item  $\mathbf{C}$ is a conjugation in $L^2(\h)$ such that $\mathbf{C}\mathbf{M}_z=\mathbf{M}_z\mathbf{C}$,
\item there is $\mathbf{U}\in L^\infty(L(\h))$ such that $\mathbf{U}(z)$ is a unitary operator for almost all $z\in\mathbb{T}$, ${M}_{\mathbf{U}}$ is $\mathbf{J}^\star$--symmetric and $\mathbf{C}={M}_{\mathbf{U}}\mathbf{J}^\star=\mathbf{J}^\star {M}_{\mathbf{U}^*}$.
    \end{enumerate}

\end{theorem}
\begin{proof} 
From the equality $\mathbf{C}\mathbf{M}_z=\mathbf{M}_z\mathbf{C}$ and Proposition \ref{js} we have
\begin{equation}\label{eq71}
\mathbf{C}\mathbf{J}^\star \mathbf{M}_z=\mathbf{C} \mathbf{M}_z \mathbf{J}^\star= \mathbf{M}_z \mathbf{C} \mathbf{J}^\star.
\end{equation}
Since the linear operator $\mathbf{C} \mathbf{J}^\star $ is unitary and commutes with $\mathbf{M}_z$, then (see \cite[Theorem 3.17, Corollary 3.19]{RR}) it is equal to ${M}_{\mathbf{U}}$, where $\mathbf{U}\in L^\infty(L(\h))$
 and $\mathbf{U}(z)$ is unitary for almost all $z\in\mathbb{T}$.
 Hence $\mathbf{C}={M}_{\mathbf{U}}\mathbf{J}^\star$. Since $\mathbf{C}$ is a conjugation, for $\f\in L^2(\h)$  {a.e. on $\mathbb{T}$} we must have
\begin{equation*}
\begin{split}
\f(z)&=(\mathbf{C}^2\f)(z)=({M}_{\mathbf{U}}\mathbf{J}^\star {M}_{\mathbf{U}}\mathbf{J}^\star \f)(z)\\&=\mathbf{U}(z)J(({M}_{\mathbf{U}}\mathbf{J}^\star \f)(\bar z))=\mathbf{U}(z)J\mathbf{U}(\bar z)J\f(z),
\end{split}
\end{equation*}
 which is equivalent to $J\mathbf{U}(\bar z)J=\mathbf{U}(z)^*$   {a.e. on $\mathbb{T}$. Therefore,} condition (2) follows from Proposition \ref{p2}(2).

 The implication $(2)\Rightarrow (1)$ is easy.
 \end{proof}

 \begin{remark} \label{3.6}
 Let us note that condition (1) in Theorem \ref{8.1} does not depend on the conjugation $J$ in $\h$. Hence if $\mathbf{C}$ satisfies (1), then condition (2) is satisfied for any conjugation $J$  in $\h$. Therefore, for two different conjugations $J$ and $J^\prime$ we obtain two different unitary operator valued functions $\mathbf{U}$ and $\mathbf{U}^\prime$ such that $$\mathbf{C}={M}_{\mathbf{U}}\mathbf{J}^{\star}={M}_{\mathbf{U}^\prime}{\mathbf{J}^\prime}^{\star}.$$
 Thus $\mathbf{U}^\prime(z)=\mathbf{U}(z)V_0$ for almost all $z\in\mathbb{T}$, where $V_0$ is a unitary operator given by $V_0=JJ^\prime$. It follows that if we have, for a given conjugation $\mathbf{C}$, the operator valued function $\mathbf{U}$ determined by some conjugation $J$ in $\h$, then we can easily obtain the function $\mathbf{U}^\prime$ corresponding to any other conjugation $J^\prime$ in $\h$.
  \end{remark}

 Recall (see \cite[Theorem 2.4]{CKLP}) that in the scalar case $C$ commutes with $M_z$ if and only if $CM_{\varphi}=M_{\varphi^{\#}}C$ for all $\varphi\in L^{\infty}$. This is not necessarily true in the general case.
  \begin{remark}
  Let $\mathbf{C}$ be a conjugation in $L^2(\h)$, $\mathbf{C}=M_{\mathbf{U}}\mathbf{J}^\star$ for some $\mathbf{U}\in L^\infty(L(\h))$ such that $\mathbf{U}(z)$ is a unitary operator for almost all $z\in\mathbb{T}$ and $M_{\mathbf{U}}$ is $\mathbf{J}^\star$--symmetric.  Assume that $\mathbf{\Phi}\in L^{\infty}(L(\h))$. Then $$\mathbf{C}{M}_{\mathbf{\Phi}}={M}_{\mathbf{\Phi^{\#}}}\mathbf{C}$$ if and only if
   \begin{equation}\label{e55}\mathbf{U}(z)J \mathbf{\Phi}(\bar z)=\mathbf{\Phi}^{\#}(z)\mathbf{U}(z)J\end{equation}
   for almost  {all $z$ in $\mathbb{T}$. Indeed, this follows from the fact that a.e. on $\mathbb{T}$,}
 \begin{displaymath}
 \begin{split}
 (\mathbf{C}{M}_{\mathbf{\Phi}}\f)(z)&=({M}_{\mathbf{U}}\mathbf{J}^\star {M}_{\mathbf{\Phi}}\f)(z)=\mathbf{U}(z)(\mathbf{J}^\star {M}_{\mathbf{\Phi}}\f)(z)\\
 &=\mathbf{U}(z)J(({M}_{\mathbf{\Phi}}\f)(\bar z))=\mathbf{U}(z)J( \mathbf{\Phi}(\bar z)\f(\bar z))
 \end{split}
 \end{displaymath}
 and
 $$({M}_{\mathbf{\Phi^{\#}}}\mathbf{C}\f)(z)=({M}_{\mathbf{\Phi}^{\#}}{M}_{\mathbf{U}}\mathbf{J}^\star \f)(z)=\mathbf{\Phi}^{\#}(z)\mathbf{U}(z)J( \f(\bar z)).$$
If $\mathbf{\Phi}(z)$ is $J$--symmetric  {for almost all $z\in\mathbb{T}$}, then $J\mathbf{\Phi}(\bar z)=\mathbf{\Phi}^{\#}(z)J$  {a.e. on $\mathbb{T}$} and condition \eqref{e55} holds provided that  {for almost all $z\in\mathbb{T}$,} $\mathbf{\Phi}^{\#}(z)$ commutes with $\mathbf{U}(z)$ .
 \end{remark}

  \begin{example}
 If $\h=\mathbb{C}$ and $J(w)=\bar w$, $w\in\mathbb{C} $, then $U\in L^\infty$ and, by Proposition \ref{p2} (2), $\mathbf{J}^\star$--symmetry of ${U}$ means that ${U}(z)={U}(\bar z)$  {a.e. on $\mathbb{T}$} and we obtain \cite[Theorem 2.4]{CKLP}.
 \end{example}

 \begin{example}\label{ec2}
 Consider $\h=\mathbb{C}^2$ and the conjugation on $\mathbb{C}^2$ defined by $J_1(z_1,z_2)=(\bar z_1,\bar z_2)$. Then, by Theorem \ref{t2}, any $\mathbf{M}_z$--commuting conjugation $\mathbf{C}$ on $L^2(\mathbb{C}^2)$ has a form
 \begin{equation}\label{c22}
  \mathbf{C}= \begin{bmatrix}
    M_{\psi_1}J^\star  & M_{\psi_2}J^\star  \\
    M_{\overline\psi_2^{\#}}J^\star  & M_{\psi_4}J^\star\\
   \end{bmatrix}=\begin{bmatrix}
    M_{\psi_1}  & M_{\psi_2}  \\
    M_{\overline\psi_2^{\#}}  & M_{\psi_4}\\
   \end{bmatrix}\begin{bmatrix}
    J^{\star}  & 0  \\
    0  & J^\star\\
   \end{bmatrix}=M_\mathbf{U}\mathbf{J}^\star_1\end{equation}
   and the functions $\psi_1,\psi_2,\psi_4\in L^\infty$ satisfy conditions \eqref{14}--\eqref{16}.
 Hence in view of Theorem \ref{8.1} the unitary operator valued function is
 $\mathbf{U}=\left[   \begin{BMAT}{cc}{cc}
    {\psi_1}  & {\psi_2}  \\
    {\overline\psi_2^{\#}}  & {\psi_4}\\
   \end{BMAT} \right]$. By conditions \eqref{14}--\eqref{16} we have that $M_\mathbf{U}$ is $\mathbf{J}^\star_1$--symmetric and the operator $\mathbf{U}(z)$ is unitary  {for almost all $z\in\mathbb{T}$}.

 On the other hand, if the conjugation on $\mathbb{C}^2$ is defined by $J_2(z_1,z_2)=(\bar z_2,\bar z_1)$, then $\mathbf{J}^\star_2=\left[   \begin{BMAT}{cc}{cc}
   0& J^{\star}  \\
    J^\star &0  \\
   \end{BMAT} \right]$. The conjugation  $\mathbf{C}$ on $L^2(\mathbb{C}^2)$ has the form
 \begin{align*}
 \mathbf{C}=\begin{bmatrix}
    M_{\psi_1}  & M_{\psi_2}  \\
    M_{\overline\psi_2^{\#}}  & M_{\psi_4}
   \end{bmatrix}\begin{bmatrix}
    J^{\star}  & 0  \\
    0  & J^\star\\
   \end{bmatrix}=&  \begin{bmatrix}
    M_{\psi_1}  & M_{\psi_2}  \\
    M_{\overline\psi_2^{\#}}  & M_{\psi_4}
  \end{bmatrix} \begin{bmatrix}
    0  & 1  \\
    1  & 0\end{bmatrix}
      \begin{bmatrix}
   0& J^{\star}  \\
    J^\star &0  \end{bmatrix}\\ =&  \begin{bmatrix}
    M_{\psi_2}  & M_{\psi_1}  \\
    M_{\psi_4}  & M_{\overline{\psi}_2^{\#}}\\
   \end{bmatrix} \mathbf{J}^\star_2=M_{\mathbf{U}_2}\mathbf{J}^\star_2.
   \end{align*}
By straightforward calculations one can check that $M_{\mathbf{U}_2}$ is $\mathbf{J}^\star_2$--symmetric and $\mathbf{U}_2(z)$ is unitary {for almost all $z\in\mathbb{T}$}.
 \end{example}

The following theorem gives a characterization of all $\mathbf{M}_z$--conjugations in $L^2(\h)$.

\begin{theorem}\label{8.3}
Let $\mathbf{C}$ be an antilinear operator in $L^2(\h)$.
Then the following are equivalent
 \begin{enumerate}
 \item  $\mathbf{C}$ is a conjugation on $L^2(\h)$ such that $\mathbf{M}_z\mathbf{C}=\mathbf{C}\,\mathbf{M}_{\bar z}$,
  \item there is ${\mathbf{C}}_0\in L^\infty(LA(\h))$ such that $\mathbf{C}=A_{{\mathbf{C}_0}}$ and ${\mathbf{C}}_0(z)$ is a conjugation for almost all $z\in\mathbb{T}$,
      \item for any conjugation $J$ in $\h$ there is $\mathbf{U}\in L^\infty(L(\h))$ such that $\mathbf{U}(z)$ is a $J$--symmetric unitary operator for almost all $z\in\mathbb{T}$ and $\mathbf{C}=\widetilde{\mathbf{J}}\,M_{\mathbf{U}}=M_{\mathbf{U}^*}\widetilde{\mathbf{J}}$.
      \end{enumerate}
\end{theorem}
\begin{proof}
First we will prove that $(1)\Rightarrow(3)$. By Proposition \ref{js}~(2) we have
\begin{equation}\label{eq82}
 \mathbf{M}_z\mathbf{C}\,\widetilde{\mathbf{J}}=\mathbf{C}\,\mathbf{M}_{\bar z}\, \widetilde{\mathbf{J}}=\mathbf{C}\,\widetilde{\mathbf{J}}\,\mathbf{M}_z,
\end{equation}
which implies that the unitary operator $\mathbf{C}\widetilde{\mathbf{J}}$ commutes with $\mathbf{M}_z$. Hence, as in the proof of Theorem \ref{8.1}, $\mathbf{C}\widetilde{\mathbf{J}}=M_{\mathbf{U}}$, for $\mathbf{U}\in L^\infty(L(\h))$ such that $\mathbf{U}(z)$ is unitary for almost all $z\in\mathbb{T}$. Then $\mathbf{C}=M_{\mathbf{U}}\widetilde{\mathbf{J}}$. Since $\mathbf{C}$ is a conjugation,  {for almost all $z\in\mathbb{T}$} we must have
\begin{displaymath}
\begin{split}
\f(z)&=(\mathbf{C}^2\,\f)(z)=(M_{\mathbf{U}}\,\widetilde{\mathbf{J}}\,M_{\mathbf{U}}\,\widetilde{\mathbf{J}}\,\f)(z)\\
&=\mathbf{U}(z)J((M_{\mathbf{U}}\,\widetilde{\mathbf{J}}\,\f)(z))=\mathbf{U}(z)J\,\mathbf{U}(z)J\f(z),
\end{split}
\end{displaymath}
which is equivalent to $J\mathbf{U}(z)J=\mathbf{U}(z)^*$
 {a.e. on $\mathbb{T}$}.

To see that $(3)\Rightarrow(2)$ define ${\mathbf{C}}_0\in L^\infty(LA(\h))$ as ${\mathbf{C}}_0(z)=\mathbf{U}(z)J$  {a.e. on $\mathbb{T}$}. An easy calculation shows that $(2)\Rightarrow(1)$.
\end{proof}
\begin{remark}\label{rem49} As in Theorem \ref{8.1}, condition (1) in Theorem \ref{8.3} does not depend on $J$. Hence (1) implies that for any conjugations $J$ and $J^\prime$ in $\h$ there exist unitary operator valued functions $\mathbf{U}$ and $\mathbf{U}^\prime$ such that
\[\mathbf{C}=M_{\mathbf{U}}\widetilde{\mathbf{J}}= M_{\mathbf{U}^\prime}{\widetilde{\mathbf{J}}}^\prime. \]
Moreover, $\mathbf{U}^\prime(z)=\mathbf{U}(z)J J^\prime $  a.e. on $\mathbb{T}$.
\end{remark}

\begin{remark}\label{rz}
Let $\mathbf{C}=A_{\mathbf{C}_0}$ be an $\mathbf{M}_z$--conjugation in $L^2(\h)$.  Suppose that $\mathbf{F}\in L^\infty(L(\h))$ is an operator valued function. Then $M_\mathbf{F}$ is $\mathbf{C}$--symmetric if and only if  {for almost all $z\in\mathbb{T}$,} $\mathbf{F}(z)$ is $\mathbf{C}_0(z)$--symmetric.
\end{remark}

\begin{example}\label{ec3}
 Consider $\h=\mathbb{C}^2$ and the conjugation $J_1$ in $\mathbb{C}^2$ defined by $J_1(z_1,z_2)=(\bar z_1,\bar z_2)$. Then by Theorem \ref{t3} any $\mathbf{M}_z$--conjugation $\mathbf{C}_1$ on $L^2(\mathbb{C}^2)$ has a form
 \begin{equation}\label{c222}
  \mathbf{C}_1=\left[   \begin{BMAT}{cc}{cc}
    M_{\psi_1}\tilde{J}  & M_{\psi_2}\tilde{J}  \\
    M_{\psi_2}\tilde{J}  & M_{\psi_4}\tilde{J}\\
   \end{BMAT} \right]=\left[   \begin{BMAT}{cc}{cc}
    M_{\psi_1}  & M_{\psi_2}  \\
    M_{\psi_2}  & M_{\psi_4}\\
   \end{BMAT} \right]\left[   \begin{BMAT}{cc}{cc}
    \tilde{J}  & 0  \\
    0  & \tilde{J}\\
   \end{BMAT} \right]=M_{\mathbf{U}}\widetilde{\mathbf{J}}_1\end{equation}
   and the functions $\psi_1,\psi_2,\psi_4\in L^\infty$ satisfy conditions \eqref{18}--\eqref{19}.
 Hence in view of Theorem \ref{8.3}, the unitary operator valued function is
 $$\mathbf{U}=\left[   \begin{BMAT}{cc}{cc}
    {\psi_1}  & {\psi_2}  \\
    {\psi_2}  & {\psi_4}\\
   \end{BMAT} \right].$$

  If we consider another conjugation in $\mathbb{C}^2$ given by $J_2(z_1,z_2)=(\bar z_2,\bar z_1)$, then 
the conjugation  $\mathbf{C}_2$ in $L^2(\mathbb{C}^2)$ has a form
 \begin{align*}
 \mathbf{C}_2=
\left[   \begin{BMAT}{cc}{cc}
    M_{\psi_2}  & M_{\psi_1}  \\
    M_{\psi_4}  & M_{{\psi}_2}\\
   \end{BMAT} \right]\left[   \begin{BMAT}{cc}{cc}
   0& \tilde{J}  \\
    \tilde{J} &0  \\
   \end{BMAT} \right]=M_{\mathbf{U}_2}\widetilde{\mathbf{J}}_2.
   \end{align*}
 \end{example}

\section{Conjugations preserving $H^2(\h)$}

Let $J$ be any conjugation in $\h$. Then {, as noted before, by \eqref{jdef} we have} $\mathbf{J}^\star H^2(\h)\subset H^2(\h)$. Now our aim is to characterize all $\mathbf{M}_z$--commuting conjugations with this property.  {Observe here that if $\mathbf{C}$ is a conjugation in $L^2(\h)$ such that $\mathbf{C}( H^2(\h))\subset H^2(\h)$, then actually we must have $\mathbf{C}(H^2(\h))= H^2(\h)$.}

\begin{theorem}\label{8.2}
Let $J$ be a conjugation in $\h$ and let $\mathbf{C}$ be a conjugation in $L^2(\h)$ such that $\mathbf{C}\mathbf{M}_z=\mathbf{M}_z\mathbf{C}$. If $\mathbf{C} (H^2(\h))\subset H^2(\h)$, then there is a unitary $J$--symmetric operator  $U_0\in L(\h)$ such that $\mathbf{C}={M}_{\mathbf{U}} \mathbf{J}^{\star}= \mathbf{J}^{\star}{M}_{\mathbf{U}^*}$, where $\mathbf{U}$ is a constant operator valued function $\mathbf{U}(z)=U_0$ for almost all $z\in\mathbb{T}$.
\end{theorem}
\begin{proof}
By Theorem \ref{8.1},  {there exists a unitary valued $\mathbf{U}\in L^{\infty}(L(\h))$ such that} $\mathbf{C}= M_{\mathbf{U}}\mathbf{J}^\star$ and $M_{\mathbf{U}}\mathbf{M}_z=\mathbf{M}_zM_{\mathbf{U}}$.  {Since, by assumption, $ M_{\mathbf{U}}\mathbf{J}^\star(H^2(\h))\subset H^2(\h)$, we have} $M_{\mathbf{U}} H^2(\h)\subset H^2(\h)$. {By the commutativity relation $${M}_\mathbf{U} \mathbf{M}_{z^k}=\mathbf{M}_{z^k}{M}_\mathbf{U}$$
we get that $$M_{\mathbf{U}} (z^k H^2(\h))\subset z^kH^2(\h)$$
 for every nonnegative integer $k$.
Since $\mathbf{C}=\mathbf{C}^{\sharp}=\mathbf{J}^\star M_{\mathbf{U}^*}$, similarly,
$$M_{\mathbf{U}^*}( z^k H^2(\h))\subset z^k H^2(\h),$$ which implies that all subspaces
$z^k H^2(\h)$, $k=0,1,\dots$, are reducing for $M_{\mathbf{U}}$. Hence also $H^2(\h)\ominus zH^2(\h)$ is reducing for $M_{\mathbf{U}}$.}
 Therefore $$M_{\mathbf{U}}\h\subset\h,$$ which means that $\mathbf{U}$ is a constant operator valued function, $\mathbf{U}(z)=U_0$ for almost all $z\in\mathbb{T}$, where $U_0 {\in L(\h)}$ is a unitary operator. Since ${M}_{\mathbf{U}}$ is $\mathbf{J}^{\star}$--symmetric, we get that $U_0$ is $J$--symmetric.
\end{proof}

\begin{remark}
If $\h=\mathbb{C}$ and $J(w)=\bar w$, $w\in\mathbb{C} $, then $U_0$ is a constant of modulus $1$ (see \cite[Corollary 3.1]{CKLP}).
\end{remark}

\begin{example}
As in Example \ref{ec2} let $\h=\mathbb{C}^2$ and $J_1(z_1,z_2)=(\bar z_1,\bar z_2)$. Then, by Theorem \ref{8.2}, Example \ref{ec2} and conditions \eqref{14}--\eqref{16}, an antilinear operator
 $\mathbf{C}=\left[   \begin{BMAT}{cc}{cc}
     D_1 & D_2 \\
     D^\sharp_2 & D_4\\
   \end{BMAT} \right]$, $i=1,2,4$, is a conjugation commuting with $\mathbf{M}_z$ and $\mathbf{C}(H^2(\mathbb{C}^2))\subset H^2(\mathbb{C}^2)$ if and only if
   $D_i=\lambda_i J^\star$ with $\lambda_i\in\mathbb{D}$ such that
   \begin{align}
   |\lambda_1|^2+|\lambda_2|^2=|\lambda_2|^2+|\lambda_4|^2=1\\ \bar\lambda_1\lambda_2+\bar\lambda_2\lambda_4=0.
   \end{align}
\end{example}

\begin{remark}
Note that if $\mathbf{C}$ is an $\mathbf{M}_z$--commuting conjugation such that $\mathbf{C} (H^2(\h))\subset H^2(\h)$, then for every conjugation $J$ in $\h$ there is a unitary $J$--symmetric operator  $U_0\in L(\h)$ such that $\mathbf{C}=\mathbf{M}_{\mathbf{U}} \mathbf{J}^{\star}$, where $\mathbf{U}$ is a constant operator valued function $\mathbf{U}(z)=U_0$ for almost all $z\in\mathbb{T}$. In view of Remark \ref{3.6} the relation between the values of constant operator valued functions corresponding to conjugations $J_1$ and $J_2$ is given by $U_2=U_1 J_1J_2$.
\end{remark}

Considering $\mathbf{M}_z$--conjugations preserving $H^2(\h)$ note  that $\widetilde{\mathbf{J}}(H^2(\h))\not\subset H^2(\h)$. In fact, $\widetilde{\mathbf{J}}(H^2(\h))= L^2(\h)\ominus zH^2(\h)$. More generally, we have:

\begin{proposition}\label{5.5}
There are no $\mathbf{M}_z$--conjugations on $L^2(\h)$ for which $H^2(\h)$ is invariant.
\end{proposition}
\begin{proof}
Assume that $\mathbf{C}$ is an $\mathbf{M}_z$--conjugation and $\mathbf{C} H^2(\h)\subset H^2(\h)$. Then $\mathbf{C} H^2(\h)= H^2(\h)$. By Theorem \ref{8.3}(3),
$\mathbf{C}=M_{\mathbf{U}}\widetilde{\mathbf{J}}$, hence $$M_{\mathbf{U}}H^2(\h)=\mathbf{C}\widetilde{\mathbf{J}}H^2(\h)=L^2(\h)\ominus zH^2(\h).$$

{Let $\f\in H^2(\h)$. Then also $z^k\f\in H^2(\h)$ for any $k=1,2,\dots$. Since $\mathbf{U}$ is an operator valued function, $M_{\mathbf{U}}$  commutes with $\mathbf{M}_z$ and ${\mathbf{M}}_{\bar z}$. Hence  $M_{\mathbf{U}}z^k\f=z^kM_{\mathbf{U}}\f\perp zH^2(\h)$   and     $M_{\mathbf{U}}\f\perp \bar z^{k-1}H^2(\h)$. Since $k$ is arbitrary, then $M_{\mathbf{U}}\f=0$, so we also have $M_{\mathbf{U}}=0$, which is a contradiction.}
%

\end{proof}

\section{Conjugations and model spaces}
Assume now that $\dim \h=d<\infty$. A function $\Theta\in H^2(L(\h))$ is called  {\it inner}, if its boundary values are unitary operators in $L^2(\h)$ almost everywhere on $\mathbb{T}$ (since $\dim\h<\infty$).
Suppose that $\Theta$ is a {\it pure} inner function, i.e., $\|\Theta(0)\|<1$. Define the corresponding {\it model space} $K_\Theta=H^2(\h)\ominus \Theta H^2(\h)$ and let $P_\Theta$ be the orthogonal projection on $\Kt$. Note that the subspace  $K^\infty_\Theta$  of all bounded functions in $K_\Theta$ ($K_\Theta^\infty=K_\Theta \cap L^\infty(\h)$) is dense in $K_\Theta$.

\begin{lemma}\label{lemat61}
	Let $\mathbf{F}\in L^{\infty}(L(\h))$ and let $J_1$, $J_2$ be any two conjugations in $\h$. Then $\widetilde{\mathbf{F}}$ defined  {a.e. on $\mathbb{T}$} by $\widetilde{\mathbf{F}}(z)=J_1\mathbf{F}(z)J_2$ belongs to $H^{\infty}(L(\h))$ if and only if $\mathbf{F}^{*}$ belongs to $H^{\infty}(L(\h))$.
\end{lemma}
\begin{proof}
 {Recall that for two Hilbert--Schmidt operators $A,B\in L(\h)$ their Hilbert--Schmidt inner product is given by
	$$\langle A,B\rangle=\tr(B^*A)=\sum_{e\in \mathcal{E}}\langle B^*Ae,e\rangle=\sum_{e\in \mathcal{E}}\langle Ae,Be\rangle,$$
	where $\mathcal{E}$ is an arbitrary orthonormal basis for $\h$ (here finite). Therefore, in our case for each $A,B\in L(\h)$,
	\begin{displaymath}
	\begin{split}
	\langle J_1AJ_2,B\rangle&=\sum_{e\in \mathcal{E}}\langle J_1AJ_2e,Be\rangle=\sum_{e\in \mathcal{E}}\langle J_1Be,AJ_2e\rangle\\
	&=\sum_{e'\in J_2(\mathcal{E})}\langle J_1BJ_2e',Ae'\rangle=\langle J_1BJ_2,A\rangle.
	\end{split}
	\end{displaymath}
It follows that if $\mathbf{F}\in L^{\infty}(L(\h))$ is given by \eqref{jeden}, then
\begin{equation}\label{zero}
\widetilde{\mathbf{F}}=\sum_{n=-\infty}^{\infty}J_1{F}_{-n}J_2\, e_n.
\end{equation}
Indeed, if $\displaystyle{\widetilde{\mathbf{F}}=\sum\limits_{n=-\infty}^{\infty}\widetilde{F}_n e_n}$, then for each $n\in\mathbb{Z}$ and $A\in L(\h)$,
\begin{displaymath}
\begin{split}
\langle \widetilde{F}_n,A\rangle&=\int_{\mathbb{T}}\langle \widetilde{\mathbf{F}}(z),A\rangle \overline{z}^n dm(z)=\int_{\mathbb{T}}\langle J_1\mathbf{F}(z)J_2,A\rangle \overline{z}^n dm(z)\\
&=\int_{\mathbb{T}}\langle J_1AJ_2,\mathbf{F}(z)\rangle \overline{z}^n dm(z)=\overline{\int_{\mathbb{T}}\langle \mathbf{F}(z),J_1AJ_2\rangle {z}^n dm(z)}\\
&=\overline{\langle F_{-n},J_1AJ_2\rangle}=\langle J_1F_{-n}J_2,A\rangle.
\end{split}
\end{displaymath}
On the other hand, if $\mathbf{F}\in L^{\infty}(L(\h))$ is given by \eqref{jeden}, then
\begin{equation}\label{zeroro}
{\mathbf{F}}^*=\sum_{n=-\infty}^{\infty}({F}_{-n})^*e_n.
\end{equation}
Indeed, if $\displaystyle{\mathbf{F}^*=\sum\limits_{n=-\infty}^{\infty}G_n e_n }$, then, using the fact that for Hilbert--Schmidt operators $A$ and $B$, $$\langle A,B\rangle=\tr(B^*A)=\overline{\tr(BA^*)}=\overline{\langle A^*,B^*\rangle}$$ (see \cite[p. 90]{conway2}), we get
\begin{displaymath}
\begin{split}
\langle G_n,A\rangle&=\int_{\mathbb{T}}\langle {\mathbf{F}}(z)^*,A\rangle \overline{z}^n dm(z)=\int_{\mathbb{T}} \overline{\langle {\mathbf{F}}(z),A^*\rangle} \overline{z}^n dm(z)\\
&=\overline{\langle F_{-n},A^*\rangle}=\langle (F_{-n})^*,A\rangle
\end{split}
\end{displaymath}
for all $A\in L(\h)$ and $n\in\mathbb{Z}$. By \eqref{zero}, $\widetilde{\mathbf{F}}\in H^{\infty}(L(\h))$ if and only if $F_n=0$ for all $n\geqslant 1$, which by \eqref{zeroro}, is equivalent to $\mathbf{F}^{*}\in H^{\infty}(L(\h))$.}
\end{proof}

Note that for $\mathbf{F}\in L^{\infty}(L(\h))$ we have that $\mathbf{F}\in H^{\infty}(L(\h))$ if and only if $M_{\mathbf{F}}(H^2(\h))\subset H^2(\h)$.

\begin{lemma}[see \cite{berc}, pp. 118-119]\label{lemat62}
	Let $\Theta\in H^{\infty}(L(\h))$ be an operator valued inner function and let $\mathbf{F}\in L^{\infty}(L(\h))$. If $M_{\mathbf{F}}(H^2(\h))\subset \Theta H^2(\h)$, then $\mathbf{F}\in H^{\infty}(L(\h))$ and there exists $\Psi\in H^{\infty}(L(\h))$ such that $\mathbf{F}=\Theta \Psi$.
\end{lemma}


\begin{proof}
	The proof follows the reasoning presented in \cite[pp. 118-119]{berc}. The inclusion $M_{\mathbf{F}}(H^2(\h))\subset \Theta H^2(\h)$ means that for each $\mathbf{f}\in H^2(\h)$ there exists $\mathbf{g}\in H^2(\h)$ such that $$M_{\mathbf{F}}\mathbf{f}=M_{\Theta}\mathbf{g}.$$ Moreover, since $M_{\Theta}$ is {an isometry}, there is only one such $\mathbf{g}$ and $$\|\mathbf{g}\|=\|M_{\Theta}\mathbf{g}\|=\|M_{\mathbf{F}}\mathbf{f}\|\leqslant \|\mathbf{F}\|_{\infty} \|\mathbf{f}\|.$$
	We can thus define a bounded linear operator $T \colon  H^2(\h)\rightarrow H^2(\h)$ by $\mathbf{f}\mapsto T\mathbf{f}=\mathbf{g}$. Hence, for $\mathbf{f}\in H^2(\h)$ we have $M_{\mathbf{F}}\mathbf{f}=M_{\Theta}T\mathbf{f}$ and
	$$M_{\Theta}T\mathbf{M}_z\mathbf{f}=M_{\mathbf{F}}\mathbf{M}_z\mathbf{f}=\mathbf{M}_zM_{\mathbf{F}}\mathbf{f}=\mathbf{M}_zM_{\Theta}T\mathbf{f}
=M_{\Theta}\mathbf{M}_zT\mathbf{f},$$
	which means that $T\mathbf{M}_z\mathbf{f}=\mathbf{M}_zT\mathbf{f}$. It follows that $T=M_{\Psi}$ for some $\Psi\in H^{\infty}(L(\h))$ (\cite[Chap. 5, Theorem 1.7]{berc}) and $M_{\mathbf{F}}\f=M_{\Theta\Psi}\f$. In particular, for every $x\in \h$,
	$$\mathbf{F}(z)x=\Theta(z)\Psi(z)x\quad  {\text{a.e. on }\mathbb{T}}$$
	and so $\mathbf{F}\in H^{\infty}(L(\h))$.
\end{proof}

 We can now give another proof of Lemma \ref{lemat61}. Namely, if $\widetilde{\mathbf{F}}\in H^{\infty}(L(\h))$, then $M_{\widetilde{\mathbf{F}}}(zH^2(\h))\subset zH^2(\h)$. Since $\mathbf{F}(z)=J_1 \widetilde{\mathbf{F}}(z)J_2$  {a.e. on $\mathbb{T}$}, we have that $M_{\mathbf{F}}=\widetilde{\mathbf{J}}_1M_{\widetilde{\mathbf{F}}}\widetilde{\mathbf{J}}_2$ and
\begin{displaymath}
\begin{split}
M_{\mathbf{F}}(H^2(\h)^{\perp})&=\widetilde{\mathbf{J}}_1M_{\widetilde{\mathbf{F}}}\widetilde{\mathbf{J}}_2(H^2(\h)^{\perp})=\widetilde{\mathbf{J}}_1M_{\widetilde{\mathbf{F}}}(zH^2(\h))\\&\subset \widetilde{\mathbf{J}}_1(zH^2(\h))=H^2(\h)^{\perp}.
\end{split}
\end{displaymath}
 It follows that $M_{\mathbf{F}^{*}}=M_{\mathbf{F}}^{*}$ preserves $H^2(\h)$ and $\mathbf{F}^{*}\in H^{\infty}(L(\h))$ by Lemma \ref{lemat62}. The proof of the other implication is analogous.

\begin{lemma}\label{lemat63}
	Let $\mathbf{F}\in H^{\infty}(L(\h))$ and let $J$ be a conjugation in $\h$. Then the following are equivalent:
	\begin{enumerate}
		\item $\mathbf{F}(z)$ is $J$--symmetric a.e. on $\mathbb{T}$;
		\item $\mathbf{F}(\lambda)$ is $J$--symmetric for all $\lambda\in\mathbb{D}$.
	\end{enumerate}
\end{lemma}
\begin{proof}
	For the proof it is enough to note that each of the conditions (1) and (2) is equivalent to $J$--symmetry of all of the coefficients {${F}_n$} of the Fourier/Taylor expansion of $\mathbf{F}$.
\end{proof}

In what follows, if the operator valued inner function $\Theta$ and the conjugation $J$ satisfy one of the conditions (1) or (2) we will simply say that $\Theta$ is  {{\it $J$--symmetric}}.

The {\it model operator} $S_\Theta\in L(\Kt)$ is given by
\begin{equation}\label{eq91}
   {(S_\Theta \f)=P_\Theta(\mathbf{M}_z\f)} \quad \text{ for } \f\in \Kt.
\end{equation}


From \cite[Theorem 3.1]{CFT} it follows that the model operator $S_\Theta$ is complex symmetric if and only if there is a conjugation $J$ in $\h$ such that for all $\lambda\in \mathbb{D}$ the matrix $\Theta(\lambda)$ is $J$--symmetric (which implies that $\Theta(z)$ is $J$--symmetric a.e. on $\mathbb{T}$).
 In that case $\Theta(z)J$ is a conjugation in $\h$  {for almost all $z\in\mathbb{T}$} and as a consequence  {$M_\Theta \widetilde{\mathbf{J}}$} is a conjugation in $L^2(\h)$ by Theorem \ref{8.3}.
From now on let us assume that $\Theta(\lambda)$ is $J$--symmetric for  {all} $\lambda\in \mathbb{D}$.
  Then $\mathbf{C}_{\Theta,J}$ defined by
  $$\mathbf{C}_{\Theta,J} \f(z)=\Theta(z)\bar z (\widetilde{\mathbf{J}}\f)(z)=\Theta(z)\bar z J(\f(z))\quad  {\text{a.e. on }\mathbb{T}} $$
 {(that is, $\mathbf{C}_{\Theta,J}=M_{\Theta}\mathbf{M}_{\overline{z}}\widetilde{\mathbf{J}}$)} is a conjugation on $L^2(\h)$ that leaves $\Kt$ invariant. Moreover, the following is true:

 \begin{proposition}\label{64}
  Let $\mathbf{V}\in L^\infty(L(\h))$ be a unitary operator valued function. Then $\mathbf{C}=M_{\mathbf{V}} \mathbf{C}_{\Theta,J}$ is a conjugation in $L^2(\h)$ if and only if $M_\mathbf{V}$ is $\mathbf{C}_{\Theta,J}$--symmetric. In other words,
 \begin{equation}\label{vzero}
 \Theta(z)J\mathbf{V}(z)\Theta(z)J=\mathbf{V}(z)^*
 \end{equation}
 almost everywhere on $\mathbb{T}$.
Moreover, if  {for almost all $z\in\mathbb{T}$,} $\mathbf{V}(z)=V_0$ {, where $V_0$} is a unitary operator in $L(\h)$, then $\Kt$ is invariant for $\mathbf{C}$.
 \end{proposition}
 \begin{proof}%
 Note that \eqref{vzero} is a consequence of $\mathbf{C}$ being an involution.
 For the proof of the second statement note that by \eqref{vzero},
 $$V_0\Theta(z)=\Theta(z)JV_0^* J\quad {\text{a.e. on }\mathbb{T}}.$$
 So for $n\geqslant 1$  {and $x\in \h$,
 \begin{displaymath}
 \begin{split}
 M_{\mathbf{V}}\mathbf{C}_{\Theta,J}( e_{-n}x)&=\mathbf{C}_{\Theta,J}M_{\mathbf{V}^*}( e_{-n}x)=\mathbf{C}_{\Theta,J}( e_{-n}V_0^*x)\\
 &=M_{\Theta}\mathbf{M}_{\overline{z}}\widetilde{\mathbf{J}}( e_{-n}V_0^*x)=M_{\Theta}(e_{n-1}JV_0^*x).
\end{split}
\end{displaymath}}
Hence $\mathbf{C}(L^2(\h)\ominus H^2(\h))\subset \Theta H^2(\h)$. On the other hand, for $n\geqslant 0$  {and $x\in \h$,
 \begin{displaymath}
 	\begin{split}
M_{\mathbf{V}} \mathbf{C}_{\Theta,J}M_\Theta (e_nx )&=M_{\mathbf{V}} M_{\Theta^*}\mathbf{C}_{\Theta,J} (e_nx )=M_{\mathbf{V}} M_{\Theta^*}M_{\Theta}\mathbf{M}_{\overline{z}}\widetilde{\mathbf{J}} (e_nx )\\
&=M_{\mathbf{V}} (e_{-n-1}Jx )=e_{-n-1}V_0Jx.
 	\end{split}
 \end{displaymath}}
 Hence $\mathbf{C}(\Theta H^2(\h))\subset L^2(\h)\ominus H^2(\h)$.
 \end{proof}

 {For $\lambda\in\mathbb{D}$ define an operator valued function $k_\lambda^\Theta$ by
$$k_\lambda^\Theta(z)=\frac{1}{1-\bar \lambda z}(1-\Theta(z)\Theta(\lambda)^*)\quad \text{a.e. on }\mathbb{T}.$$
For each $x\in \h$ denote the function $z\mapsto k_\lambda^\Theta(z)x$ by $k_\lambda^\Theta x$. Recall from \cite{KT} that $k_\lambda^\Theta x\in K_{\Theta}$ and for each $f\in K_{\Theta}$,
$$\langle f,k_\lambda^\Theta x\rangle=\langle f(\lambda),x\rangle$$
(the inner product on the left is the $L^2(\h)$ inner product while the inner product on the right is the inner product from $\h$). Similarly, denote by $\widetilde{k_\lambda^\Theta}x$ the function $z\mapsto \widetilde{k_\lambda^\Theta}(z)x$, where
$$\widetilde{k_\lambda^\Theta}(z)=\frac{1}{z-\lambda}(\Theta(z)-\Theta(\lambda))\quad \text{a.e. on }\mathbb{T}.$$}

For $\Theta\in H^2(L(\h))$ recall that $\Theta^{\#}(z)=\Theta(\bar z)^*$, and $\Theta$ is inner if and only if $\Theta^{\#}$ is inner.

\begin{lemma}\label{l66}
	Let $J$ be a conjugation on $\h$ such that  {$\Theta$ is $J$--symmetric.  Then $\mathbf{C}_{\Theta,J}( k_\lambda^\Theta x)=\widetilde{k_\lambda^\Theta}Jx$} for $x\in \h$.
\end{lemma}
\begin{proof}
	Note that
	\begin{align*}\label{eq911}
	(\mathbf{C}_{\Theta,J}\, k_\lambda^\Theta x)(z) &= \Theta(z)\bar z J(k_\lambda^\Theta(z)x)=\Theta(z)\bar zJ(\tfrac{1}{1-\bar\lambda z}(1-\Theta(z)\Theta(\lambda)^*)x) \\
	&= \tfrac{\bar z}{1-\lambda \bar z}\Theta(z)(1-\Theta(z)^*\Theta(\lambda))Jx=\tfrac{1}{z-\lambda}(\Theta(z)-\Theta(\lambda))Jx.
	\end{align*}
\end{proof}


 Proposition \ref{64} describes a class of conjugations in $L^2(\h)$ which leave model spaces invariant. The following result says that amongst all $\mathbf{M}_z$--conjugations only conjugations in that class have this property.
\begin{theorem}\label{66}
Let $\mathbf{C}$ be an $\mathbf{M}_z$--conjugation on $L^2(\h)$ and let $\Theta\in L^\infty(L(\h))$ be a pure inner function such that $\Theta(z)$ is $J$--symmetric for almost all $z\in \mathbb{T}$ with a conjugation $J$ on $\h$. Suppose that $\Kt$ is invariant for $\mathbf{C}$. Then $\mathbf{C}=M_\mathbf{V}\mathbf{C}_{\Theta,J}$ with $\mathbf{V}$ a unitary valued  constant function, $\mathbf{V}(z)=V_0$  {a.e. on $\mathbb{T}$,}  such that $M_\mathbf{V}$ is $\mathbf{C}_{\Theta,J}$--symmetric.
\end{theorem}
\begin{proof}
By Theorem \ref{8.3} we know that $\mathbf{C}=M_\mathbf{U}\widetilde{\mathbf{J}}=\widetilde{\mathbf{J}}M_{\mathbf{U}^*}$, where $\mathbf{U}\in L^\infty(L(\h))$,  {and for almost all $z\in\mathbb{T}$,} $\mathbf{U}(z)$ is unitary and $J$--symmetric. Fix $x\in\h$. Note that  {a.e. on $\mathbb{T}$,}
\begin{align*}
(\mathbf{C}\mathbf{C}_{\Theta,J}\, k_\lambda^\Theta x)(z)&
=\mathbf{U}(z)J(\tfrac{1}{z-\lambda}(\Theta(z)-\Theta(\lambda)))Jx\\&
=\mathbf{U}(z)\tfrac{1}{\bar z-\bar \lambda}(\Theta(z)^*-\Theta(\lambda)^*)x\\&=\tfrac{z}{1-\bar\lambda z}\mathbf{U}(z)\Theta(z)^*(1-\Theta(z)\Theta(\lambda)^*)x\\&=z\mathbf{U}(z)\Theta(z)^*k_\lambda^\Theta(z)x.
\end{align*}
Denote by $\mathbf{W}_1$ the operator valued function given by $$\mathbf{W}_1(z)=z\mathbf{U}(z)\Theta(z)^*(1-\Theta(z)\Theta(0)^*)\quad {\text{a.e. on }\mathbb{T}}.$$
 Since $\mathbf{C}\mathbf{C}_{\Theta,J}\, k_0^\Theta x\in \Kt$, then $M_{\mathbf{W}_1}x \in H^2(\h)$.
By commutativity of $M_{\mathbf{W}_1}$ with $\mathbf{M}_z$, we get that for  {$n=0,1,\dots$,
$$M_{\mathbf{W}_1} (e_n x)\in H^2(\h).$$}
Since $x\in \h$ is arbitrary, hence $H^2(\h)$ is invariant for $M_{\mathbf{W}_1}$.

 {Recalling that $\Theta$ is a pure analytic function, i.e., $\|\Theta(0)\|<1$, we have that $z\mapsto (1-\Theta(z)\Theta(0)^*)^{-1}$ is a bounded analytic function. Hence for any $\f\in H^2(\h)$ we have $(1-\Theta(\cdot)\Theta(0)^*)^{-1}\f\in H^2(\h)$ and
$$M_{\mathbf{W}_1}(1-\Theta(\cdot)\Theta(0)^*)^{-1}\f=zM_\mathbf{U}\Theta(\cdot)^*\f\in H^2(\h).$$}
Therefore, $\mathbf{V}=\mathbf{M}_z \mathbf{U}\Theta^*$ is analytic.
On the other hand,  {a.e. on $\mathbb{T}$,}
\begin{align*}
(\mathbf{C}_{\Theta,J} \mathbf{C} k_0^\Theta x)(z)&=(\mathbf{C}_{\Theta,J} \widetilde{\mathbf{J}}M_{\mathbf{U}^*}k_0^\Theta x)(z)\\&
=\Theta(z)\bar z \mathbf{U}(z)^* (1-\Theta(z)\Theta(0)^*)x\\&=\Theta(z)\mathbf{U}(z)^*\bar z(1-\Theta(z)\Theta(0)^*)x.
\end{align*}
Similarly to what was done above we define $\mathbf{W}_2(z)=\Theta(z)\mathbf{U}(z)^*\bar z(1-\Theta(z)\Theta(0)^*)$  {a.e. on $\mathbb{T}$}. Observe that
 $M_{\mathbf{W}_2} \mathbf{M}_z=\mathbf{M}_zM_{\mathbf{W}_2}$, and consequently as above $M_{\mathbf{W}_2}( H^2(\h))\subset H^2(\h)$. Hence for $\f\in H^2(\h)$ we get $$M_{\mathbf{W}_2}(1-\Theta(\cdot)\Theta(0)^*)^{-1}\f= {M_\Theta M_{\mathbf{U}^*}\mathbf{M}_{\bar z}} \f\in H^2(\h).$$
 {Therefore, $\mathbf{V}^*=\Theta\mathbf{U}^*\mathbf{M}_{\bar z}$ is also analytic.}
It follows that $\mathbf{V}$ is a constant unitary operator valued function, $\mathbf{V}(z)=V_0$  {a.e. on $\mathbb{T}$}. A direct calculation shows that $V_0$ satisfies  \eqref{vzero}. Hence $\mathbf{C}=M_\mathbf{V}\mathbf{C}_{\Theta,J}$.
\end{proof}

\begin{remark}
Let $\Theta\in L^\infty(L(\h))$ be a pure inner function and let $\mathbf{C}$ be an
$\mathbf{M}_z$--conjugation in $L^2(\h)$ which leaves $\Kt$ invariant. Then by Theorem \ref{66}, for every conjugation $J$ in $\h$ such that $\Theta(z)$ is $J$--symmetric a.e. on $\mathbb{T}$, there exists a unitary operator $V_0\in L(\h)$ such that $ \mathbf{C}=M_{\mathbf{V}}\mathbf{C}_{\Theta,J}$, where $\mathbf{V}(z)=V_0$  {a.e. on $\mathbb{T}$}. Therefore,  for such conjugations $J$ and $J^\prime$ in $\h$ there exist unitary operators $V_0,V_0^\prime$ and constant operator valued functions $\mathbf{V}(z)=V_0$, $\mathbf{V}^\prime(z)=V_0^\prime$  {(a.e. on $\mathbb{T}$)} such that
\[\mathbf{C}=M_{\mathbf{V}}\mathbf{C}_{\Theta,J}= M_{\mathbf{V}^\prime}\mathbf{C}_{\Theta,J^\prime}. \]
Moreover, $V_0^\prime= V_0 J J^\prime $.
\end{remark}

\begin{example}
{To illustrate Theorem \ref{66} consider $\h=\mathbb{C}^2$ and the conjugation $J_1(z_1,z_2)=(\bar z_1,\bar z_2)$. Note that $\Theta(z)=\begin{bmatrix} z& 0\\0& z^2\end{bmatrix}$ is $J_1$--symmetric and defines a pure inner matrix valued function. Then $\Kt=\{(a_0,b_0+b_1 z): a_0, b_0, b_1\in\mathbb{C}\}$ and the conjugation $\mathbf{C}_{\Theta,J_1}$ is equal to
$\begin{bmatrix} \tilde{J}&0\\0&z\tilde{J}\end{bmatrix}$.}

{Assume that $\mathbf{C}$ is an $\mathbf{M}_z$--conjugation. By Theorem \ref{t3}
$$\mathbf{C}=\left[   \begin{BMAT}{cc}{cc}
    M_{\psi_1}\tilde{J}  & M_{\psi_2}\tilde{J}  \\
    M_{\psi_2}\tilde{J}  & M_{\psi_4}\tilde{J}\\
   \end{BMAT} \right]$$
with $\psi_i$ satisfying \eqref{18}--\eqref{19}.
 Hence $$\mathbf{C}(a_0,b_0+b_1 z)=(\psi_1 \bar a_0+\psi_2(\bar b_0+\bar b_1\bar z),\psi_2\bar a_0+\psi_4(\bar b_0+\bar b_1\bar z))$$
  for any $a_0, b_0, b_1\in \mathbb{C}$. If $\mathbf{C}$  leaves subspace $\Kt$ invariant, then
 $\psi_1=\lambda_1$, $\psi_2=0$, $\psi_4=\lambda_4 z$ with $\lambda_i\in\mathbb{T}$. It is clear that $\mathbf{C}=\begin{bmatrix}                                                                                                                       \lambda_1 & 0 \\
 0 & \lambda_4
  \end{bmatrix}\mathbf{C}_{\Theta,J_1}$.}
\end{example}


\section{Conjugations between model spaces}
In this section we consider conjugations which map one model space into another.
Let $\dim\h<\infty$.

\begin{lemma}\label{lemat64}
	Let $\Theta,\Lambda\in H^{\infty}(L(\h))$ be two inner functions. Then the following are equivalent:
	\begin{enumerate}
		\item $\Lambda^{*}\Theta\in H^{\infty}(L(\h))$;
		\item $\Theta H^2(\h)\subset \Lambda H^2(\h)$;
		\item $K_{\Lambda}\subset K_{\Theta}$.
	\end{enumerate}
\end{lemma}
\begin{proof}
	If (1) holds, that is, $\Lambda^{*}\Theta=\Psi\in H^{\infty}(L(\h))$, then $\Theta=\Lambda\Psi$ and so (2) is satisfied. On the other hand, if (2) is satisfied, then by Lemma \ref{lemat62}, $\Theta=\Lambda\Psi$ for some $\Psi\in H^{\infty}(L(\h))$, $\Psi=\Lambda^{*}\Theta$ and (1) follows. The equivalence of (2) and (3) is obvious.
\end{proof}

In view of Lemma \ref{lemat64} we will say that an operator valued inner function $\Lambda$ divides an operator valued inner function $\Theta$, and we write $\Lambda\leqslant \Theta$, if $\Lambda^{*}\Theta\in H^{\infty}(L(\h))$. Equivalently, $\Lambda\leqslant \Theta$, if $\Theta=\Lambda\Psi$ for some $\Psi\in H^{\infty}(L(\h))$ (clearly  $\Psi=\Lambda^{*}\Theta$ is also an operator valued inner function).

\begin{proposition}\label{stw65}
	Let $\Theta,\Lambda\in H^{\infty}(L(\h))$ be two inner functions. If there exists a conjugation $J$ in $\h$ such that both $\Theta$ and $\Lambda$ are $J$--symmetric, then the following are equivalent:
	\begin{enumerate}
		\item $\Lambda^{*}\Theta\in H^{\infty}(L(\h))$;
		\item $\Theta\Lambda^{*}\in H^{\infty}(L(\h))$.
	\end{enumerate}
\end{proposition}
\begin{proof}
 Assume that $\Theta\Lambda^*\in H^\infty(L(\h))$. Then $\Theta\Lambda^*(H^2(\h))\subset H^2(\h)$. Note also that $\Theta \Lambda^*(\Lambda H^2(\h))=\Theta H^2(\h)$ and $$\mathbf{C}_{\Theta,J}=M_{\Theta\Lambda^{*}}\mathbf{C}_{\Lambda,J}.$$
Since both $\Lambda$ and $\Theta$ are $J$--symmetric,  {we have} $\mathbf{C}_{\Lambda,J}(K_\Lambda)=K_\Lambda$ and $\mathbf{C}_{\Theta,J}(K_\Theta)=K_\Theta$.  Hence $$\mathbf{C}_{\Theta,J}(K_\Lambda)= M_{\Theta\Lambda^{*}}\mathbf{C}_{\Lambda,J}(K_\Lambda)=M_{\Theta\Lambda^{*}}(K_\Lambda)\subset K_\Theta,$$
which implies that
$$K_\Lambda\subset \mathbf{C}_{\Theta,J}(K_\Theta)=K_\Theta.$$
By Lemma \ref{lemat64} it follows that $\Lambda^*\Theta\in H^\infty(L(\h))$. The other implication can be proved analogously.
\end{proof}

\begin{theorem}\label{th66}
	Let $\Theta,\Lambda\in H^{\infty}(L(\h))$ be two pure inner functions and assume that there exist conjugations $J_{\Theta}$, $J_{\Lambda}$ in $\h$ such that $\Theta$ is $J_{\Theta}$--symmetric and $\Lambda$ is $J_{\Lambda}$--symmetric. Moreover, let $\mathbf{C}$ be an $\mathbf{M}_z$--conjugation in $L^2(\h)$. Then $\mathbf{C}(K_{\Lambda})\subset \Kt$ if and only if $\mathbf{C}=\mathbf{C}_{\Gamma,J}$ for some inner function $\Gamma\in H^{\infty}(L(\h))$ and for some conjugation $J$ in $\h$ such that $\Lambda\leqslant \Gamma\leqslant\Theta$ and $\Gamma$ is $J$--symmetric. In particular, then $\Lambda\leqslant\Theta$.
\end{theorem}
\begin{proof}
	Assume first that $\mathbf{C}=\mathbf{C}_{\Gamma,J}$, $\Gamma$ is $J$--symmetric and $\Lambda\leqslant \Gamma\leqslant\Theta$. Then $K_{\Lambda}\subset K_{\Gamma}\subset K_{\Theta}$ and
	$$\mathbf{C}(K_{\Lambda})=\mathbf{C}_{\Gamma,J}(K_{\Lambda})\subset \mathbf{C}_{\Gamma,J}(K_{\Gamma})=K_{\Gamma}\subset K_{\Theta}.$$
	
	Assume now that $\mathbf{C}$ is an $\mathbf{M}_z$--conjugation such that $\mathbf{C}(K_{\Lambda})\subset \Kt$.
	By Theorem \ref{8.3} there exist $\mathbf{U}_{\Lambda},\mathbf{U}_{\Theta}\in L^{\infty}(\h)$ such that
	 $$\mathbf{C}=M_{\mathbf{U}_{\Lambda}}\widetilde{\mathbf{J}}_{\Lambda}=\widetilde{\mathbf{J}}_{\Theta}M_{\mathbf{U}_{\Theta}^{*}},$$ the function $\mathbf{U}_{\Lambda}$ is unitary valued and $J_{\Lambda}$--symmetric, and the function $\mathbf{U}_{\Theta}$ is unitary valued and $J_{\Theta}$--symmetric. Moreover, it follows from the Remark \ref{rem49} that
	 \begin{equation}\label{ut}
	   \mathbf{U}_{\Theta}(z)=\mathbf{U}_{\Lambda}(z)J_{\Lambda}J_{\Theta}\quad \text{a.e. on }\mathbb{T}.
	 \end{equation}
 By Lemma \ref{l66}, for $x\in\h$ we have  {a.e. on $\mathbb{T}$,}
	\begin{align*}
	(\mathbf{C}\mathbf{C}_{\Lambda,J_{\Lambda}} k_0^\Lambda x)(z)&
	=\mathbf{U}_{\Lambda}(z)J_{\Lambda}((\mathbf{C}_{\Lambda,J_{\Lambda}} k_0^\Lambda x)(z))\\
	&=\mathbf{U}_{\Lambda}(z)J_{\Lambda}(\bar{z}(\Lambda(z)-\Lambda(0))J_{\Lambda}x)\\
	&=\mathbf{U}_{\Lambda}(z)z(\Lambda(z)^*-\Lambda(0)^*)x\\
	&=\mathbf{U}_{\Lambda}(z)z\Lambda(z)^*(1-\Lambda(z)\Lambda(0)^*)x.
	\end{align*}
	Define
	$$\mathbf{V}_1(z)=\mathbf{U}_{\Lambda}(z)z\Lambda(z)^*\quad \text{and}\quad \mathbf{W}_1(z)=\mathbf{V}_1(z)(1-\Lambda(z)\Lambda(0)^*)$$
	 {(a.e. on $\mathbb{T}$).} Clearly, $\mathbf{V}_1, \mathbf{W}_1\in L^{\infty}(L(\h))$.  {By the calculations above, $M_{\mathbf{W}_1}(e_0x)=\mathbf{C}\mathbf{C}_{\Lambda, J_\Lambda} k_0^\Lambda x\in H^2(\h)$. Since $M_{\mathbf{W}_1}$ commutes with $\mathbf{M}_z$, we also get $M_{\mathbf{W}_1} (e_n x)\in H^2(\h)$ for $n=0,1,\dots$.} Hence $M_{\mathbf{W}_1}(H^2(\h))\subset H^2(\h)$ and $\mathbf{W}_1\in H^{\infty}(L(\h))$.	
	
	{Recalling that $\Lambda$ is a pure analytic function, we have that $z\mapsto (1-\Lambda(z)\Lambda(0)^*)^{-1}$ belongs to $H^{\infty}(L(\h))$. Hence also $\mathbf{V}_1\in H^{\infty}(L(\h))$ and $\mathbf{V}_1$ is inner, since its values are unitary operators.}
	
	On the other hand,  using \eqref{ut} we get  {(a.e. on $\mathbb{T}$)}
	\begin{align*}
	(\mathbf{C}_{\Theta,J_{\Theta}} \mathbf{C} k_0^\Lambda x)(z)&=(\mathbf{C}_{\Theta,J_{\Theta}} \widetilde{\mathbf{J}}_{\Theta}M_{\mathbf{U}_{\Theta}^*}k_0^\Lambda x)(z)\\
	&=\Theta(z)\bar z \mathbf{U}_{\Theta}(z)^* (1-\Lambda(z)\Lambda(0)^*)x\\
	&=\Theta(z)\bar z J_{\Theta}J_{\Lambda}\mathbf{U}_{\Lambda}(z)^*(1-\Lambda(z)\Lambda(0)^*)x.
	\end{align*}
	Define
	$$\mathbf{V}_2(z)=\Theta(z)\bar z J_{\Theta}J_{\Lambda}\mathbf{U}_{\Lambda}(z)^* \quad \text{and}\quad \mathbf{W}_2(z)=\mathbf{V}_2(z)(1-\Lambda(z)\Lambda(0)^*)$$
	 {(a.e. on $\mathbb{T}$).} As above $H^2(\h)$ is $M_{\mathbf{W}_2}$--invariant so $\mathbf{W}_2\in H^{\infty}(L(\h))$, and consequently $\mathbf{V}_2\in H^{\infty}(L(\h))$  is an inner function. Since
	$$\mathbf{V}_2(z)=J_{\Theta}\Theta(z)^* z \mathbf{U}_{\Lambda}(z)J_{\Lambda}\quad  {\text{a.e. on }\mathbb{T})},$$
	it follows from Lemma \ref{lemat61} that $\mathbf{V}_3(z)=\mathbf{U}_{\Lambda}(z)^* \bar z \Theta(z)$  {(a.e. on $\mathbb{T}$)} is also an inner function.
	
	Define $$\Gamma(z)=\mathbf{U}_{\Lambda}(z)z\quad  {\text{a.e. on }\mathbb{T})}.$$ Then $\mathbf{V}_1=\Gamma\Lambda^*\in H^{\infty}(L(\h))$, and so $\Gamma=\mathbf{V}_1\Lambda$ is an inner function. Observe also that $\Gamma$ is $J_{\Lambda}$--symmetric. By Proposition \ref{stw65} we also have $\Lambda^*\Gamma\in H^{\infty}(L(\h))$, and so $\Lambda\leqslant \Gamma$.
	
	Moreover,   $\mathbf{V}_3=\Gamma^*\Theta\in H^{\infty}(L(\h))$ and $\Gamma\leqslant \Theta$. In particular,
	$$\Lambda^*\Theta=\Lambda^*\Gamma\Gamma^*\Theta\in H^{\infty}(L(\h)),$$
	that is, $\Lambda\leqslant \Theta$. Finally,  {a.e. on $\mathbb{T}$},
	$$(\mathbf{C}f)(z)=\mathbf{U}_{\Lambda}(z)J_{\Lambda}(f(z))=\Gamma(z)\bar z J_{\Lambda}(f(z))=(\mathbf{C}_{\Gamma, J_{\Lambda}}f)(z).$$
\end{proof}

\begin{remark}\label{rem6}
	Note that in the proof of Theorem \ref{th66} we actually did not use the fact that $\Theta$ is pure. Moreover, if we assume that $\Lambda$ and $\Theta$ are $J$--symmetric, it follows form the proof that $\Gamma$ can be chosen to be $J$--symmetric as well.
\end{remark}

\begin{remark}\label{rem7}
	Let $\mathbf{C}=\mathbf{C}_{\Gamma, J}$ for an inner function $\Gamma\in H^{\infty}(L(\h))$, which is $J$--symmetric and such that $\Lambda\leqslant \Gamma\leqslant\Theta$. Then for any conjugation $J'$ in $\h$ the function $\Gamma'$ defined by $$\Gamma'(z)=\Gamma(z) J J'\quad {\text{a.e. on }\mathbb{T}}$$
	is $J'$--symmetric, $\Lambda\leqslant\Gamma'\leqslant \Theta$ and $\mathbf{C}=\mathbf{C}_{\Gamma', J'}$.
\end{remark}

\begin{remark}\label{rem8}
	If $\Lambda=\Theta$ is $J$--symmetric and $\mathbf{C}$ is an  $\mathbf{M}_z$--conjugation in $L^2(\h)$ such that $\mathbf{C}(K_{\Theta})\subset \Kt$, then by Theorem \ref{th66} there exists an inner $J$--symmetric function $\Gamma$ such that $\mathbf{C}=\mathbf{C}_{\Gamma, J}$ and $\Theta\leqslant\Gamma\leqslant\Theta$. The last condition implies that $\mathbf{V}=\Gamma\Theta^*$ is a unitary constant and  {a.e. on $\mathbb{T}$,}
	$$(\mathbf{C}f)(z)=\Gamma(z)\bar z J(f(z))=\Gamma(z)\Theta(z)^*\Theta(z)\bar z J(f(z))=(M_{\mathbf{V}}\mathbf{C}_{\Theta,J}f)(z).$$
	Moreover, $\mathbf{V}$ satisfies \eqref{vzero} and $M_{\mathbf{V}}$ is $\mathbf{C}_{\Theta,J}$--symmetric.
\end{remark}

Now  recall that $\mathbf{J}^{\star}$ is an $\mathbf{M}_z$--commuting conjugation. The proposition below shows some basic properties of $\mathbf{J}^{\star}$ as to model spaces.

\begin{proposition}\label{pr77}
  Let $J$ be a conjugation in $\h$ and let $\Theta$ be a $J$--symmetric inner function. Then
\begin{enumerate}
	\item $\mathbf{J}^{\star}M_{\Theta}=M_{\Theta^{\#}}\mathbf{J}^{\star}$;
		\item $\mathbf{J}^{\star}(\Theta H^2(\h))=\Theta^{\#}H^2(\h)$;
				\item $\mathbf{J}^{\star}(K_{\Theta})=K_{\Theta^{\#}}$,
				\item  {$\mathbf{J}^{\star}(k_0^{\Theta}x)=k_0^{\Theta^{\#}}Jx$.}
\end{enumerate}
\end{proposition}
 In what follows we describe all $\mathbf{M}_z$--commuting conjugations mapping one model space into another.
\begin{theorem}\label{th78}
	Let $\Theta,\Lambda\in H^{\infty}(L(\h))$ be two pure inner functions and assume that there exist conjugations $J_{\Theta}$, $J_{\Lambda}$ in $\h$ such that $\Theta$ is $J_{\Theta}$--symmetric and $\Lambda$ is $J_{\Lambda}$--symmetric. Assume that  $\mathbf{C}$ is an $\mathbf{M}_z$--commuting conjugation in $L^2(\h)$. Then $\mathbf{C}(K_{\Lambda})\subset \Kt$ if and only if there is a unitary $J_{\Lambda}$--symmetric operator  $U_0\in L(\h)$ such that $\mathbf{C}={M}_{\mathbf{U}_\Lambda} \mathbf{J}_\Lambda^{\star}$, where $\mathbf{U}_\Lambda$ is a constant operator valued function, $\mathbf{U}_\Lambda(z)=U_0$ for almost all $z\in\mathbb{T}$ and  $U_0\Lambda^{\#}\leqslant \Theta$.
\end{theorem}
\begin{proof}
Since $\mathbf{C}$ is $\mathbf{M}_z$--commuting, then by Theorem \ref{8.1} there is $\mathbf{U}_\Lambda\in L^\infty(L(\h))$ such that $\mathbf{U}_\Lambda(z)$ is a unitary operator for almost all $z\in\mathbb{T}$, ${M}_{\mathbf{U}_\Lambda}$ is $\mathbf{J}_\Lambda^\star$--symmetric and $\mathbf{C}={M}_{\mathbf{U}_\Lambda}\mathbf{J}_\Lambda^\star$.
Note that using Lemma \ref{l66} and Proposition \ref{pr77}~(4)  {a.e on $\mathbb{T}$} we have
\begin{equation}\label{r77}
\begin{split}
  (\mathbf{C}\mathbf{C}_{\Lambda,J_\Lambda}\,\widetilde{k_0^\Lambda} x)(z)&= \mathbf{U}_\Lambda(z)\mathbf{J}^\star_\Lambda(k_0^\Lambda J_\Lambda x)(z)\\
  &=  \mathbf{U}_\Lambda(z)(1-\Lambda(z)^{\#}(\Lambda(0)^{\#})^*)x.
  \end{split}
\end{equation} Define $\mathbf{V}(z)= \mathbf{U}_\Lambda(z)(1-\Lambda(z)^{\#}(\Lambda(0)^{\#})^*)$  {a.e on $\mathbb{T}$}. It is clear that $\mathbf{V}\in L^\infty(L(\h))$ and by \eqref{r77} we have that  {$M_{\mathbf{V}}(e_0x)=\mathbf{C}\mathbf{C}_{\Lambda,J_\Lambda}\,\widetilde{k_0^\Lambda} x\in H^2(\h)$ because  $\mathbf{C}(K_{\Lambda})\subset \Kt$. Since $M_{\mathbf{V}}$ commutes with $\mathbf{M}_z$, we have also that $M_{\mathbf{V}}(e_n x)\in H^2(\h)$ for $n=1,2,\dots$.} Hence
$M_{\mathbf{V}}(H^2(\h))\subset H^2(\h)$ and $\mathbf{V}\in H^\infty(L(\h))$.
{Since $\Lambda$ is pure, it follows that $z\mapsto (1-\Lambda(z)^{\#}(\Lambda(0)^{\#})^*)^{-1}\in H^\infty(L(\h))$. By \eqref{r77} we obtain that}
 {$$\mathbf{U}_\Lambda =\mathbf{V} (1-\Lambda^{\#}(\Lambda(0)^{\#})^*)^{-1}\in H^\infty(L(\h)),$$}
which implies that $\mathbf{C}$ leaves $H^2(\h)$ invariant.
Applying Theorem \ref{8.2} we get that $\mathbf{U}_\Lambda$ is a constant operator valued function and $\mathbf{U}_\Lambda(z)=U_0$ a.e. on $\mathbb{T}$. Since $\mathbf{U}_\Lambda$ is $\mathbf{J}_{\Lambda}^\star$--symmetric thus $U_0$ is $J$--symmetric by Proposition \ref{p2}.

Note that by Proposition \ref{pr77}~(3), we have $$\mathbf{C}(K_\Lambda)=U_0 K_{\Lambda^{\#}}.$$
Since $U_0$ is unitary, then $U_0 K_{\Lambda^{\#}}=K_{U_0\Lambda^{\#}}$.
Hence by Lemma \ref{lemat64} $$\mathbf{C}(K_\Lambda) \subset K_\Theta$$
 only if $U_0\Lambda^{\#}\leqslant \Theta$.
\end{proof}
\begin{remark}\label{th79}
Consider  a pure inner function $\Theta\in H^{\infty}(L(\h))$ and  a conjugation $J_{\Theta}$ in $\h$ such that $\Theta$ is $J_{\Theta}$--symmetric. As a consequence of Theorem \ref{th78} we have that if  $\mathbf{C}$ is an $\mathbf{M}_z$--commuting conjugation in $L^2(\h)$,  then $\mathbf{C}(\Kt)\subset {K}_{\Theta^{\#}}$ if and only if there is a unitary operator  $U_0\in L(\h)$ such that $\mathbf{C}={M}_{\mathbf{U}_\Theta} \mathbf{J}_\Theta^{\star}$, where $\mathbf{U}_\Theta$ is a constant $\mathbf{J}_\Theta^\star$--symmetric operator valued function, $\mathbf{U}_\Theta(z)=U_0$ for almost all $z\in\mathbb{T}$ and $U_0\Theta^{\#}\leqslant \Theta^{\#}$. Note that if $U_0$ commutes with $\Theta$, the last condition is always satisfied, since then $(U_0\Theta^{\#})^*\Theta^{\#}=U_0^*\in H^\infty(L(\h))$.
\end{remark}

\section{Conjugations and shift invariant subspaces}

Let $dim\, \h<\infty$ and let $\Theta,\Lambda\in H^{\infty}(L(\h))$ be two inner functions. For any fixed conjugation $J$ in $\h$ define
$$\mathbf{C}_J^{\Lambda,\Theta}=M_{\Theta}\mathbf{J}^{\star}M_{\Lambda^*}.$$
Clearly, $\mathbf{C}_J^{\Lambda,\Theta}$ is an antilinear isometry. Moreover, it is easy to see that $\mathbf{C}_J^{\Lambda,\Theta}$ maps $\Lambda H^2(\h)$ onto $\Theta H^2(\h)$.

\begin{proposition}	\label{stw71}
The antilinear operator $\mathbf{C}_J^{\Lambda,\Theta}$ is an involution (and hence a conjugation in $L^2(\h)$) if and only if
\begin{equation}\label{x1}
\Theta(z)J\Lambda^{\#}(z)=\Lambda(z)J\Theta^{\#}(z)\quad\text{a.e. on }\mathbb{T}.
\end{equation}
\end{proposition}
\begin{proof}
	We have
	$$(\mathbf{C}_J^{\Lambda,\Theta})^2=M_{\Theta}\mathbf{J}^{\star}M_{\Lambda^*}M_{\Theta}\mathbf{J}^{\star}M_{\Lambda^*}=I_{L^2(\h)}$$
	if and only if
	\begin{equation}\label{x2}
	M_{\Theta}\mathbf{J}^{\star}M_{\Lambda^*}=M_{\Lambda}\mathbf{J}^{\star}M_{\Theta^*}.
	\end{equation}
	Since for $\f\in L^2(\h)$  {a.e. on $\mathbb{T}$,}
	$$(M_{\Theta}\mathbf{J}^{\star}M_{\Lambda^*}\f)(z)=\Theta(z)J(\Lambda(\bar z)^*\f(\bar z))=(M_{\Theta J\Lambda^{\#}}\widetilde{\mathbf{J}}\mathbf{J}^{\star}\f)(z)$$
	and
	$$(M_{\Lambda}\mathbf{J}^{\star}M_{\Theta^*}\f)(z)=\Lambda(z)J(\Theta(\bar z)^*\f(\bar z))=(M_{\Lambda J\Theta^{\#}}\widetilde{\mathbf{J}}\mathbf{J}^{\star}\f)(z),$$
	we see that \eqref{x2} is equivalent to \eqref{x1}.
\end{proof}

For $\mathbf{F}\in L^{\infty}(L(\h))$ and a conjugation $J$ in $\h$ define $$\mathbf{F}_J(z)=J\mathbf{F}(\bar z)J,\quad  {\text{a.e. on }\mathbb{T}}.$$
Clearly, $\mathbf{F}_J\in  L^{\infty}(L(\h))$. Moreover, $\mathbf{F}_J\in  H^{\infty}(L(\h))$ if and only if $\mathbf{F}\in H^{\infty}(L(\h))$ (see {the proof of} Lemma \ref{lemat61}), and  $\mathbf{F}_J$ is an inner function if and only if $\mathbf{F}$ is. It is easy to verify the following.

\begin{lemma}\label{L72}
	Let $\mathbf{F},\mathbf{G}\in L^{\infty}(L(\h))$ and let $J$ be a conjugation in $\h$. Then
	\begin{enumerate}
		\item $(\mathbf{F}_J)_J=\mathbf{F}$;		
		\item $(\mathbf{F}\mathbf{G})_J=\mathbf{F}_J\mathbf{G}_J$;
		\item  {$(\mathbf{F}_J)^{*}(z)=J\mathbf{F}^{\#}(z)J=(\mathbf{F}^{*})_J(z)$ a.e. on $\mathbb{T}$};
		\item  {$(\mathbf{F}_J)^{\#}(z)=J\mathbf{F}^{*}(z)J=(\mathbf{F}^{\#})_J(z)$ a.e. on $\mathbb{T}$}.
	\end{enumerate}
\end{lemma}

Moreover, from the proof of Proposition \ref{p2} we get:

\begin{lemma}\label{L73}
		Let $\mathbf{F}\in L^{\infty}(L(\h))$ and let $J$ be a conjugation in $\h$. Then
	$$\mathbf{J}^{\star}M_{\mathbf{F}}\mathbf{J}^{\star}=M_{\mathbf{F}_J}.$$
\end{lemma}

\begin{remark}\label{uw73}
	Note that the condition \eqref{x1} can be expressed as
	\begin{equation}\label{x3}
	\Theta\Lambda_J^{*}=\Lambda\Theta_J^{*}.
	\end{equation}
	Indeed, \eqref{x1} is equivalent to
	$$\Theta(z)J\Lambda^{\#}(z)J=\Lambda(z)J\Theta^{\#}(z)J\quad  {\text{a.e. on }\mathbb{T}},$$
	and by Lemma \ref{L72}\,(3),  {for almost all $z\in\mathbb{T}$,} $J\Lambda^{\#}(z)J=\Lambda_J^{*}(z)$ and $J\Theta^{\#}(z)J=\Theta_J^{*}(z)$.
\end{remark}

\begin{remark}
	If $\h=\mathbb{C}$ and $J(w)=\bar w$, $w\in\mathbb{C}$, then for $\varphi\in L^{\infty}(\mathbb{T})$ we have $$\varphi_J=\varphi^{\#}.$$ Moreover, for scalar inner functions $\theta$ and $\alpha$ the condition \eqref{x1} (or \eqref{x3}) takes form $$\theta(z)\alpha(\bar z)=\alpha(z)\theta(\bar z)\quad {\text{a.e. on }\mathbb{T}},$$ which in this case is equivalent to $\theta\theta^{\#}=\alpha\alpha^{\#}$ (see \cite[Theorem 5.2]{CKLP}).
\end{remark}

\begin{theorem}\label{75}
	Let $\Theta,\Lambda\in H^{\infty}(L(\h))$ be two inner functions and let $J$ be a conjugation in $\h$. There exists an  $\mathbf{M}_z$--commuting conjugation $\mathbf{C}$ in $L^2(\h)$ such that $\mathbf{C}(\Lambda H^2(\h))\subset \Theta H^2(\h)$  if and only if there is an inner function $\Psi\in H^{\infty}(L(\h))$ such that
	\begin{equation}\label{x4}
	(\Psi\Lambda^{*}\Theta)_J=(\Psi\Lambda^{*}\Theta)^{*}.
	\end{equation}
	In that case, $\mathbf{C}=\mathbf{C}_J^{\Lambda,\Gamma}$ for some inner function $\Gamma\in H^{\infty}(L(\h))$ such that $\Theta\leqslant \Gamma$ and
	 	\begin{equation}\label{x5}
	 	\Gamma\Lambda_J^{*}=\Lambda\Gamma_J^{*}.
	 	\end{equation}
\end{theorem}
\begin{proof}
	Assume first that there exists an inner function $\Psi\in H^{\infty}(L(\h))$ such that \eqref{x4} holds. Then
	$$\Theta^{*}\Lambda \Psi^{*}=(\Psi\Lambda^{*}\Theta)^{*}=(\Psi\Lambda^{*}\Theta)_J=\Psi_J\Lambda_J^{*}\Theta_J,$$
	and so
	\begin{equation}\label{x6}
		\Lambda \Psi^{*}\Theta_J^{*}=\Theta\Psi_J\Lambda_J^{*}.
	\end{equation}
	Put $\Gamma =\Theta\Psi_J$. Then $\Gamma\in H^{\infty}(L(\h))$ is an inner function, $\Theta\leqslant \Gamma$ and by \eqref{x6},
	$$\Lambda \Gamma_J^{*}=\Lambda(\Theta_J\Psi)^{*}=\Lambda\Psi^{*}\Theta_J^{*}=\Theta\Psi_J\Lambda_J^{*}=\Gamma\Lambda_J^{*}.$$
	Thus \eqref{x5} holds and, by Proposition \ref{stw71} and Remark \ref{uw73}, $\mathbf{C}=\mathbf{C}_J^{\Lambda,\Gamma}$ is a conjugation in $L^2(\h)$. Moreover, it is an $\mathbf{M}_z$--commuting conjugation such that $$\mathbf{C}(\Lambda H^2(\h))=\Gamma H^2(\h) \subset \Theta H^2(\h).$$
	
	Assume now that $\mathbf{C}$ is an $\mathbf{M}_z$--commuting conjugation in $L^2(\h)$ such that $\mathbf{C}(\Lambda H^2(\h))\subset \Theta H^2(\h)$. By Theorem \ref{8.1} and Proposition  \ref{p2}~(2) there exists a unitary valued $\mathbf{U}\in L^\infty(L(\h))$ such that $$\mathbf{C}={M}_{\mathbf{U}}\mathbf{J}^\star=\mathbf{J}^\star {M}_{\mathbf{U}^*}$$
	and  {$J\mathbf{U}(z)J=\mathbf{U}^{\#}(z)$ a.e. on $\mathbb{T}$}. Therefore,
	$$\mathbf{C}(\Lambda H^2(\h))=\mathbf{J}^\star {M}_{\mathbf{U}^*}M_{\Lambda}( H^2(\h))\subset M_{\Theta} (H^2(\h)),$$
	and by Lemma \ref{L73},
	$$ {M}_{\mathbf{U}^*\Lambda}( H^2(\h))\subset \mathbf{J}^\star M_{\Theta}( H^2(\h))= M_{\Theta_J}\mathbf{J}^\star (H^2(\h))=\Theta_J H^2(\h).$$
	By Lemma \ref{lemat62}, $\mathbf{U}^*\Lambda\in H^{\infty}(L(\h))$ and there exists an inner function $\Psi\in H^{\infty}(L(\h))$ such that $\mathbf{U}^*\Lambda=\Theta_J \Psi$. Note that by  {the fact that $J\mathbf{U}(z)J=\mathbf{U}^{\#}(z)$ a.e. on $\mathbb{T}$} and Lemma \ref{L72}~(3), {we get} $\mathbf{U}_J^{*}=\mathbf{U}$. It follows that
$$\Theta\Psi_J\Lambda_J^{*}=(\Theta_J\Psi\Lambda^{*})_J=(\mathbf{U}^{*}\Lambda\Lambda^{*})_J=\mathbf{U}=\Lambda\Psi^{*}\Theta_J^{*},$$	
which is an equivalent form of \eqref{x4}. Moreover, the above means that the function $\Gamma=\Theta \Psi_J=\mathbf{U}_J^*\Lambda_J$ satisfies \eqref{x5} (since $\Gamma_J^{*}=\Psi^{*}\Theta_J^{*}$). Clearly, $\Gamma\in H^{\infty}(L(\h))$ is an inner function and $\Theta\leqslant \Gamma$. Moreover,
$$\mathbf{C}={M}_{\mathbf{U}}\mathbf{J}^\star={M}_{\mathbf{U}_J^{*}}\mathbf{J}^\star{M}_{\Lambda}{M}_{\Lambda^{*}}={M}_{\mathbf{U}_J^{*}\Lambda_J}\mathbf{J}^\star{M}_{\Lambda^{*}}={M}_{\Gamma}\mathbf{J}^\star{M}_{\Lambda^{*}}=\mathbf{C}_J^{\Lambda,\Gamma}.$$
\end{proof}

\begin{remark}
Let $\Theta\in H^{\infty}(L(\h))$ be an inner function. Assume that $\mathbf{C}$ is an $\mathbf{M}_z$--commuting conjugation in $L^2(\h)$ such that $\mathbf{C}(\Theta H^2(\h)) \subset \Theta H^2(\h)$.
By Theorem \ref{75} we obtain that $\mathbf{C}=\mathbf{C}_J^{\Theta,\Gamma}$ for some inner function $\Gamma\in  H^{\infty}(L(\h))$ such that $\Theta\leqslant \Gamma$ and
	\begin{equation}\label{5}
	\Gamma\Theta_J^{*}=\Theta\Gamma_J^{*}.
	\end{equation}

	Therefore there exists $\Psi\in H^{\infty}(L(\h))$ such that $\Gamma=\Theta\Psi$ and by \eqref{5},
	$$\Theta\Psi\Theta_J^{*}=\Theta\Psi_J^{*}\Theta_J^{*}.$$
	It follows that $\Psi=\Psi_J^{*}$. Since $\Psi_J, \Psi^*_J\in H^\infty(L(\h))$,  so $\Psi$ must be a unitary constant. Assume that $\Psi(z)=U_0\in L(\h)$ a.e. on $\mathbb{T}$, then  $\Gamma(z)=\Theta(z)U_0$  {a.e. on $\mathbb{T}$} and
	$$\mathbf{C}=M_{\Gamma}\mathbf{J}^{\star}M_{\Theta^{*}}=M_{\Theta}M_{U_0}\mathbf{J}^{\star}M_{\Theta^{*}}=M_{\Theta U_0\Theta ^{*}}M_{\Theta}\mathbf{J}^{\star}M_{\Theta^{*}}.$$
Note that by \eqref{5} we now have
\begin{equation*}
\Theta(z)U_0 J\Theta(\bar z)^* J=\Theta(z)JU_0^*\Theta(\bar z)^*J\quad  {\text{a.e. on }\mathbb{T}},
\end{equation*}
which implies that
$U_0 J=JU_0^*$, i.e., $U_0$ is $J$--symmetric. Recalling the scalar case considered in \cite[Corollary 5.4]{CKLP} one can expect that $\Theta U_0\Theta ^{*}$ is a unitary constant. This is not necessarily true (see Example \ref{ex88}).
\end{remark}
\begin{example}\label{ex88}
{Let $\h=\mathbb{C}^2$ and consider the conjugation $J(z_1,z_2)=(\bar z_1,\bar z_2)$ in $\mathbb{C}^2$.   If we take $\Theta=\begin{bmatrix}
                                           1 & 0 \\
                                           0 & z
                                         \end{bmatrix}$ and
 $U_0=\begin{bmatrix} 0&1\\1&0
\end{bmatrix}$.  It is easy to see that  both $\Theta$ and $U_0$ are $J$--symmetric, but  $\Theta U_0\Theta^*=\begin{bmatrix} 0&\bar z\\z&0
\end{bmatrix}$, so it is not constant.}
\end{example}

\end{document}